\documentclass[a4paper,10pt]{amsart}
\usepackage{booktabs}

\usepackage[utf8]{inputenc}
\usepackage{url}

\usepackage[expansion=false]{microtype}
\usepackage{amssymb,amsmath,amsopn,amsxtra,amsthm,amsfonts}
\usepackage{xr}
\usepackage[mathcal]{euscript}
\usepackage{mathalfa}
\usepackage{txfonts}
\usepackage{todonotes}

\DeclareMathAlphabet{\mathbb}{U}{msb}{m}{n}

\usepackage[pdfusetitle,unicode]{hyperref}

\numberwithin{equation}{section}
\makeatletter
\let\c@equation\c@subsection

\makeatother

\newtheorem{cor}[subsection]{Corollary}
\newtheorem{lem}[subsection]{Lemma}
\newtheorem{prop}[subsection]{Proposition}

\newtheorem{thm}[subsection]{Theorem}

\theoremstyle{definition}
\newtheorem{defn}[subsection]{Definition}
\newtheorem{rem}[subsection]{Remark}
\newtheorem{constr}[subsection]{Construction}
\newtheorem{ex}[subsection]{Example}

\theoremstyle{remark}

\renewcommand{\eqref}[1]{(\ref{#1})}

\usepackage{tikz}
\usetikzlibrary{matrix}
\usepackage{tikz-cd}
\tikzset{shorten <>/.style={shorten >=#1,shorten <=#1}}
\tikzcdset{arrow style=tikz,diagrams={>={Straight Barb[scale=0.8]}}}

\usepackage{fixltx2e}

\usepackage{xspace}

\usepackage{xifthen}

\usepackage{xparse}

\newcommand{\nc}{\newcommand}
\nc{\renc}{\renewcommand}
\nc{\ssec}{\subsection}
\nc{\sssec}{\subsubsection}
\nc{\on}{\operatorname}
\nc{\term}[1]{#1\xspace}

\DeclareMathSymbol{A}{\mathalpha}{operators}{`A}
\DeclareMathSymbol{B}{\mathalpha}{operators}{`B}
\DeclareMathSymbol{C}{\mathalpha}{operators}{`C}
\DeclareMathSymbol{D}{\mathalpha}{operators}{`D}
\DeclareMathSymbol{E}{\mathalpha}{operators}{`E}
\DeclareMathSymbol{F}{\mathalpha}{operators}{`F}
\DeclareMathSymbol{G}{\mathalpha}{operators}{`G}
\DeclareMathSymbol{H}{\mathalpha}{operators}{`H}
\DeclareMathSymbol{I}{\mathalpha}{operators}{`I}
\DeclareMathSymbol{J}{\mathalpha}{operators}{`J}
\DeclareMathSymbol{K}{\mathalpha}{operators}{`K}
\DeclareMathSymbol{L}{\mathalpha}{operators}{`L}
\DeclareMathSymbol{M}{\mathalpha}{operators}{`M}
\DeclareMathSymbol{N}{\mathalpha}{operators}{`N}
\DeclareMathSymbol{O}{\mathalpha}{operators}{`O}
\DeclareMathSymbol{P}{\mathalpha}{operators}{`P}
\DeclareMathSymbol{Q}{\mathalpha}{operators}{`Q}
\DeclareMathSymbol{R}{\mathalpha}{operators}{`R}
\DeclareMathSymbol{S}{\mathalpha}{operators}{`S}
\DeclareMathSymbol{T}{\mathalpha}{operators}{`T}
\DeclareMathSymbol{U}{\mathalpha}{operators}{`U}
\DeclareMathSymbol{V}{\mathalpha}{operators}{`V}
\DeclareMathSymbol{W}{\mathalpha}{operators}{`W}
\DeclareMathSymbol{X}{\mathalpha}{operators}{`X}
\DeclareMathSymbol{Y}{\mathalpha}{operators}{`Y}
\DeclareMathSymbol{Z}{\mathalpha}{operators}{`Z}

\nc{\sA}{\ensuremath{\mathcal{A}}\xspace}
\nc{\sB}{\ensuremath{\mathcal{B}}\xspace}
\nc{\sC}{\ensuremath{\mathcal{C}}\xspace}
\nc{\sD}{\ensuremath{\mathcal{D}}\xspace}
\nc{\sE}{\ensuremath{\mathcal{E}}\xspace}
\nc{\sF}{\ensuremath{\mathcal{F}}\xspace}
\nc{\sG}{\ensuremath{\mathcal{G}}\xspace}
\nc{\sH}{\ensuremath{\mathcal{H}}\xspace}
\nc{\sI}{\ensuremath{\mathcal{I}}\xspace}
\nc{\sJ}{\ensuremath{\mathcal{J}}\xspace}
\nc{\sK}{\ensuremath{\mathcal{K}}\xspace}
\nc{\sL}{\ensuremath{\mathcal{L}}\xspace}
\nc{\sM}{\ensuremath{\mathcal{M}}\xspace}
\nc{\sN}{\ensuremath{\mathcal{N}}\xspace}
\nc{\sO}{\ensuremath{\mathcal{O}}\xspace}
\nc{\sP}{\ensuremath{\mathcal{P}}\xspace}
\nc{\sQ}{\ensuremath{\mathcal{Q}}\xspace}
\nc{\sR}{\ensuremath{\mathcal{R}}\xspace}
\nc{\sS}{\ensuremath{\mathcal{S}}\xspace}
\nc{\sT}{\ensuremath{\mathcal{T}}\xspace}
\nc{\sU}{\ensuremath{\mathcal{U}}\xspace}
\nc{\sV}{\ensuremath{\mathcal{V}}\xspace}
\nc{\sW}{\ensuremath{\mathcal{W}}\xspace}
\nc{\sX}{\ensuremath{\mathcal{X}}\xspace}
\nc{\sY}{\ensuremath{\mathcal{Y}}\xspace}
\nc{\sZ}{\ensuremath{\mathcal{Z}}\xspace}

\nc{\bA}{\ensuremath{\mathbf{A}}\xspace}
\nc{\bB}{\ensuremath{\mathbf{B}}\xspace}
\nc{\bC}{\ensuremath{\mathbf{C}}\xspace}
\nc{\bD}{\ensuremath{\mathbf{D}}\xspace}
\nc{\bE}{\ensuremath{\mathbf{E}}\xspace}
\nc{\bF}{\ensuremath{\mathbf{F}}\xspace}
\nc{\bG}{\ensuremath{\mathbf{G}}\xspace}
\nc{\bH}{\ensuremath{\mathbf{H}}\xspace}
\nc{\bI}{\ensuremath{\mathbf{I}}\xspace}
\nc{\bJ}{\ensuremath{\mathbf{J}}\xspace}
\nc{\bK}{\ensuremath{\mathbf{K}}\xspace}
\nc{\bL}{\ensuremath{\mathbf{L}}\xspace}
\nc{\bM}{\ensuremath{\mathbf{M}}\xspace}
\nc{\bN}{\ensuremath{\mathbf{N}}\xspace}
\nc{\bO}{\ensuremath{\mathbf{O}}\xspace}
\nc{\bP}{\ensuremath{\mathbf{P}}\xspace}
\nc{\bQ}{\ensuremath{\mathbf{Q}}\xspace}
\nc{\bR}{\ensuremath{\mathbf{R}}\xspace}
\nc{\bS}{\ensuremath{\mathbf{S}}\xspace}
\nc{\bT}{\ensuremath{\mathbf{T}}\xspace}
\nc{\bU}{\ensuremath{\mathbf{U}}\xspace}
\nc{\bV}{\ensuremath{\mathbf{V}}\xspace}
\nc{\bW}{\ensuremath{\mathbf{W}}\xspace}
\nc{\bX}{\ensuremath{\mathbf{X}}\xspace}
\nc{\bY}{\ensuremath{\mathbf{Y}}\xspace}
\nc{\bZ}{\ensuremath{\mathbf{Z}}\xspace}

\nc{\dA}{\ensuremath{\mathds{A}}\xspace}
\nc{\dB}{\ensuremath{\mathds{B}}\xspace}
\nc{\dC}{\ensuremath{\mathds{C}}\xspace}
\nc{\dD}{\ensuremath{\mathds{D}}\xspace}
\nc{\dE}{\ensuremath{\mathds{E}}\xspace}
\nc{\dF}{\ensuremath{\mathds{F}}\xspace}
\nc{\dG}{\ensuremath{\mathds{G}}\xspace}
\nc{\dH}{\ensuremath{\mathds{H}}\xspace}
\nc{\dI}{\ensuremath{\mathds{I}}\xspace}
\nc{\dJ}{\ensuremath{\mathds{J}}\xspace}
\nc{\dK}{\ensuremath{\mathds{K}}\xspace}
\nc{\dL}{\ensuremath{\mathds{L}}\xspace}
\nc{\dM}{\ensuremath{\mathds{M}}\xspace}
\nc{\dN}{\ensuremath{\mathds{N}}\xspace}
\nc{\dO}{\ensuremath{\mathds{O}}\xspace}
\nc{\dP}{\ensuremath{\mathds{P}}\xspace}
\nc{\dQ}{\ensuremath{\mathds{Q}}\xspace}
\nc{\dR}{\ensuremath{\mathds{R}}\xspace}
\nc{\dS}{\ensuremath{\mathds{S}}\xspace}
\nc{\dT}{\ensuremath{\mathds{T}}\xspace}
\nc{\dU}{\ensuremath{\mathds{U}}\xspace}
\nc{\dV}{\ensuremath{\mathds{V}}\xspace}
\nc{\dW}{\ensuremath{\mathds{W}}\xspace}
\nc{\dX}{\ensuremath{\mathds{X}}\xspace}
\nc{\dY}{\ensuremath{\mathds{Y}}\xspace}
\nc{\dZ}{\ensuremath{\mathds{Z}}\xspace}

\nc{\bbA}{\ensuremath{\mathbb{A}}\xspace}
\nc{\bbB}{\ensuremath{\mathbb{B}}\xspace}
\nc{\bbC}{\ensuremath{\mathbb{C}}\xspace}
\nc{\bbD}{\ensuremath{\mathbb{D}}\xspace}
\nc{\bbE}{\ensuremath{\mathbb{E}}\xspace}
\nc{\bbF}{\ensuremath{\mathbb{F}}\xspace}
\nc{\bbG}{\ensuremath{\mathbb{G}}\xspace}
\nc{\bbH}{\ensuremath{\mathbb{H}}\xspace}
\nc{\bbI}{\ensuremath{\mathbb{I}}\xspace}
\nc{\bbJ}{\ensuremath{\mathbb{J}}\xspace}
\nc{\bbK}{\ensuremath{\mathbb{K}}\xspace}
\nc{\bbL}{\ensuremath{\mathbb{L}}\xspace}
\nc{\bbM}{\ensuremath{\mathbb{M}}\xspace}
\nc{\bbN}{\ensuremath{\mathbb{N}}\xspace}
\nc{\bbO}{\ensuremath{\mathbb{O}}\xspace}
\nc{\bbP}{\ensuremath{\mathbb{P}}\xspace}
\nc{\bbQ}{\ensuremath{\mathbb{Q}}\xspace}
\nc{\bbR}{\ensuremath{\mathbb{R}}\xspace}
\nc{\bbS}{\ensuremath{\mathbb{S}}\xspace}
\nc{\bbT}{\ensuremath{\mathbb{T}}\xspace}
\nc{\bbU}{\ensuremath{\mathbb{U}}\xspace}
\nc{\bbV}{\ensuremath{\mathbb{V}}\xspace}
\nc{\bbW}{\ensuremath{\mathbb{W}}\xspace}
\nc{\bbX}{\ensuremath{\mathbb{X}}\xspace}
\nc{\bbY}{\ensuremath{\mathbb{Y}}\xspace}
\nc{\bbZ}{\ensuremath{\mathbb{Z}}\xspace}

\nc{\mrm}[1]{\ensuremath{\mathrm{#1}}\xspace}
\nc{\mbf}[1]{\ensuremath{\mathbf{#1}}\xspace}
\nc{\mcal}[1]{\ensuremath{\mathcal{#1}}\xspace}
\nc{\msc}[1]{\ensuremath{\mathscr{#1}}\xspace}

\renc{\bar}[1]{\overline{#1}}

\let\sectsign\S
\let\S\relax

\nc{\sub}{\subset}
\nc{\too}{\longrightarrow}
\nc{\hook}{\hookrightarrow}
\nc*{\hooklongrightarrow}{\ensuremath{\lhook\joinrel\relbar\joinrel\rightarrow}}
\nc{\hooklong}{\hooklongrightarrow}
\nc{\twoheadlongrightarrow}{\relbar\joinrel\twoheadrightarrow}
\nc{\shiso}{\approx}
\nc{\isoto}{\xrightarrow{\sim}}
\nc{\isofrom}{\xleftarrow{\sim}}
\renc{\ge}{\geqslant}
\renc{\le}{\leqslant}
\renc{\geq}{\geqslant}
\renc{\leq}{\leqslant}

\nc{\id}{\mathrm{id}}
\nc{\can}{\mathrm{can}}

\DeclareMathOperator{\rk}{\mathrm{rk}}
\DeclareMathOperator{\Hom}{\mathrm{Hom}}
\nc{\uHom}{\underline{\smash{\Hom}}}

\DeclareMathOperator{\End}{\mathrm{End}}
\DeclareMathOperator{\Sym}{\mathrm{Sym}}
\nc{\Pre}{\mathrm{PSh}{}}
\nc{\uEnd}{\underline{\smash{\End}}}

\renc{\lim}{\operatorname*{lim}}
\nc{\colim}{\operatorname*{colim}}
\nc{\Cofib}{\on{Cofib}}
\nc{\Fib}{\on{Fib}}
\nc{\initial}{\varnothing}
\nc{\op}{\mathrm{op}}

\nc{\bDelta}{\mbf{\Delta}}
\nc{\DM}{\mbf{DM}}
\nc{\eff}{\mathrm{eff}}
\nc{\veff}{\mathrm{veff}}
\nc{\hzmw}{\mathrm{H}\tilde{\Z}}
\nc{\hz}{\mathrm{H}\Z}
\nc{\cyc}{{\mrm{cyc}}}
\nc{\betah}{{\tilde{\beta}}}
\nc{\corr}{{\on{corr}}}
\nc{\fet}{{\mrm{f\acute et}}}
\nc{\fsyn}{{\mrm{fsyn}}}
\nc{\fflat}{{\mrm{fflat}}}
\nc{\syn}{{\mrm{syn}}}
\nc{\Gor}{{\mrm{Gor}}}
\nc{\gor}{{\mrm{gor}}}
\nc{\Pic}{\mrm{Pic}}
\nc{\perf}{\mrm{perf}}
\nc{\oblv}{\on{oblv}}
\nc{\exact}{\on{exact}}

\nc{\F}{{\on{F}}}
\nc{\clopen}{{\mrm{clopen}}}
\nc{\B}{\mrm{B}}
\nc{\D}{\mrm{D}}
\nc{\Fin}{\on{Fin}}
\nc{\Cut}{\on{Cut}}
\nc{\Cart}{\on{Cart}}
\nc{\pairs}{\mathsf{pairs}}
\nc{\Pairs}{\mathrm{Pair}}
\nc{\Trip}{\mathrm{Trip}}
\nc{\Lab}{\mathrm{Lab}}
\nc{\coCart}{\mathrm{coCart}}
\nc{\RKE}{\mathrm{RKE}}
\nc{\strict}{\mathrm{strict}}
\nc{\Emb}{\mathrm{Emb}}
\nc{\EMB}{\mathcal{E}\mathrm{mb}}
\nc{\Split}{\mathrm{Split}}
\nc{\Set}{\mathrm{Set}}
\nc{\sSets}{\mathrm{sSets}}
\nc{\pb}{\mathrm{pb}}
\nc{\lci}{\mathrm{lci}}
\nc{\fib}{\mathrm{fib}}
\nc{\cofib}{\mathrm{cofib}}
\nc{\diff}{\mrm{diff}}
\nc{\gp}{\mrm{gp}}
\nc{\chr}{\mrm{char}}
\nc{\mgp}{\mrm{mot-gp}}
\nc{\FSyn}{\mrm{FSyn}}
\nc{\FFlat}{{\mrm{FFlat}}}
\nc{\FEt}{\mrm{FEt}}
\nc{\Spc}{\mrm{Spc}}
\nc{\Ob}{\mrm{Ob}}
\nc{\Spt}{\mrm{Spt}}
\nc{\T}{\bT}
\nc{\suspinf}{\Sigma^\infty}
\nc{\h}{\mrm{h}}
\nc{\uhom}{\underline{\mathrm{Hom}}}
\nc{\umap}{\underline{\mathrm{Maps}}}
\renc{\H}{\bH}
\nc{\Einfty}{{\sE_\infty}}
\nc{\Eone}{{\sE_1}}
\nc{\Stab}{\mrm{Stab}}
\nc{\lax}{{\mrm{lax}}}
\nc{\cocart}{{\mrm{cocart}}}
\nc{\Sch}{\mrm{Sch}}
\nc{\dSch}{\mrm{dSch}}
\nc{\Aff}{\mrm{Aff}}
\nc{\SmAff}{\mrm{SmAff}}
\nc{\dAff}{\mrm{dAff}}
\nc{\Fr}{\on{Fr}}
\nc{\A}{\mathbf{A}}
\nc{\N}{\mathbf{N}}
\nc{\Z}{\mathbf{Z}}
\nc{\Q}{\mathbf{Q}}
\nc{\GW}{\mathbf{GW}}
\nc{\GWspace}{\mathcal{GW}}
\nc{\KGW}{\mathbb G\mathrm W}
\nc{\GWspectrum}{\mathrm{GW}}
\nc{\uGW}{\underline{\mathrm{GW}}}
\nc{\uW}{\underline{\mathrm{W}}}
\nc{\uZ}{\underline{\mathbf{Z}}}
\nc{\Oo}{\mathcal{O}} 
\nc{\red}{{\on{red}}}
\nc{\Voev}{{\on{Voev}}}
\nc{\Corr}{\mrm{Corr}}
\nc{\Span}{\mathbf{Corr}}
\nc{\Gap}{\mrm{Gap}}
\nc{\Filt}{\mrm{Filt}}
\nc{\Corrfr}{\Corr^{\fr}}
\nc{\Corrvfr}{\Corr^{\Vfr}}
\nc{\Spec}{\on{Spec}}
\nc{\Sm}{\mrm{Sm}}
\nc{\QSm}{\mrm{QSm}}
\nc{\Gm}{\mathbf{G}_{\mrm{m}}}
\renc{\P}{\bP}
\nc{\Nis}{\mathrm{Nis}}
\nc{\Zar}{\mathrm{Zar}}
\nc{\et}{\mathrm{\acute et}}
\nc{\all}{\mathrm{all}}
\nc{\fold}{\mathrm{fold}}
\nc{\Fun}{\mathrm{Fun}}
\nc{\Ho}{\mathrm{Ho}}
\nc{\Segal}{\mathrm{Segal}}
\nc{\Mon}{\mrm{Mon}{}}
\nc{\Ab}{\mrm{Ab}}
\nc{\Gr}{\mrm{Gr}}
\nc{\GrO}{\mrm{GrO}}
\nc{\Sh}{\on{Sh}}
\nc{\M}{\mrm{M}}
\nc{\Lhtp}{L_{\A^1}}
\nc{\Lmot}{L_{\mrm{mot}}}
\nc{\mot}{\mrm{mot}}
\nc{\SH}{\mbf{SH}}
\nc{\RR}{\mbf{R}}
\nc{\CC}{\mbf{C}}
\nc{\Mod}{\mrm{Mod}}
\nc{\QCoh}{\mrm{QCoh}}
\nc{\MonUnit}{\mbf{1}}
\nc{\tr}{\on{tr}}
\nc{\vop}{\mrm{vop}}
\nc{\fr}{{\on{fr}}}
\nc{\Ar}{\mrm{Ar}}
\nc{\Vfr}{\on{Vfr}}
\nc{\frdiff}{{\on{frdiff}}}
\nc{\frGys}{\on{frGys}}
\nc{\SHfr}{\SH^{\fr}}
\nc{\SHfrdiff}{\SH^{\frdiff}}
\nc{\SHfrGys}{\SH^{\frGys}}
\nc{\InftyCat}{\infty\textnormal{-}\mrm{Cat}}
\nc{\TriCat}{\mathrm{TriCat}}
\nc{\Cat}{\mathrm{1\textnormal{-}Cat}}
\nc{\Th}{\on{Th}}
\def\G{\bG}
\nc{\CMon}{\mrm{CMon}{}}
\nc{\CAlg}{\mrm{CAlg}{}}
\nc{\MGL}{\mrm{MGL}}
\nc{\PMGL}{\mathrm{PMGL}}
\nc{\KGL}{\mrm{KGL}}
\nc{\kgl}{\mrm{kgl}}
\nc{\KQ}{\mrm{KQ}}
\nc{\kq}{\mrm{kq}}
\nc{\KSp}{\mrm{KSp}}
\nc{\ksp}{\mrm{ksp}}
\nc{\MSL}{\mrm{MSL}}
\nc{\MSp}{\mrm{MSp}}
\nc{\Seg}{\mrm{Seg}{}}
\nc{\Tw}{\mrm{Tw}}
\nc{\sslash}{/\mkern-6mu/}
\nc{\PrL}{\mrm{Pr}^\mrm{L}}
\nc{\PrR}{\mrm{Pr}^\mrm{R}}
\nc{\pr}{\mrm{pr}}
\nc{\efr}{\mrm{efr}}
\nc{\nfr}{\mrm{nfr}}
\nc{\dfr}{\mrm{fr}}
\nc{\tfr}{\mrm{tfr}}
\nc{\Vect}{\mrm{Vect}}
\nc{\sVect}{\mrm{sVect}}
\nc{\fix}{\mrm{fix}}
\nc{\Hilb}{\mathrm{Hilb}}
\nc{\flci}{\mathrm{flci}}
\nc{\Isom}{\mathrm{Isom}}
\nc{\GL}{\mathrm{GL}}
\nc{\BGL}{\mathrm{BGL}}
\nc{\SL}{\mathrm{SL}}
\nc{\Sp}{\mathrm{Sp}}
\nc{\fin}{\mathrm{fin}}
\nc{\cl}{\mathrm{cl}}
\nc{\cn}{\mathrm{cn}}
\nc{\sm}{\mathrm{sm}}
\nc{\heart}{\heartsuit}
\nc{\1}{\mbf{1}}
\renc{\o}{\mrm{or}}
\nc{\GWpsh}{\mrm{GW}}
\nc{\ev}{\mrm{ev}}
\nc{\Ann}{\mrm{Ann}}
\nc{\FSYN}{\mathcal{FS}\mrm{yn}}
\nc{\FFLAT}{\mathcal{FF}\mrm{lat}}
\nc{\mrk}{\mrm{mrk}}%
\nc{\FFmrk}{\mathcal{FF}\mrm{lat}^{\mrk}}
\nc{\FFnu}{\mathcal{FF}\mrm{lat}^{\mrm{nu}}}
\nc{\FFbas}{\mathcal{FF}\mrm{lat}^{\mrm{bas}}}
\nc{\FQSM}{\mathcal{FQS}\mathrm{m}}
\nc{\Quot}{\mathrm{Quot}}    %% Quot scheme
\nc{\COH}{\mathcal{C}\mathrm{oh}}
\let\phi\varphi

\let\emptyset\varnothing
\nc{\Rees}{\mrm{Rees}}

\nc{\robber}{\mathcal{R}}%
\nc{\mv}{\mrm{mv}}
\nc{\const}{\mrm{const}}
\nc{\robbermv}{\robber^{\mv}}
\nc{\robberconst}{\robber^{\const}}
\nc{\robbernot}{\robber_0}%
\nc{\sectionmv}{i^{\mv}}%
\nc{\sectionconst}{i^{\const}}%
\nc{\sectionnot}{i}%
\nc{\st}{\mathrm{st}}

\nc{\FGor}{{\mrm{FGor}}}     %% Finite Gorenstein
\nc{\fgor}{{\mrm{fgor}}}
\nc{\Aug}{{\mrm{Aug}}}  
\nc{\Uni}{{\mrm{Uni}}}  
\nc{\ori}{{\mrm{or}}}        %% Oriented
\nc{\sym}{{\mrm{sym}}}  
\nc{\alt}{{\mrm{alt}}}   
\nc{\nonu}{{\mrm{nu}}}  
\nc{\Ima}{\operatorname{Im}} %% Image of a linear map
\nc{\Hyp}{\operatorname{Hyp}}
\nc{\tel}{\mathrm{tel}}
\nc{\FQSm}{\mathrm{FQSm}}
\nc{\Perf}{\mathrm{Perf}}

\nc{\inftyCat}{\term{$\infty$-category}}
\nc{\inftyCats}{\term{$\infty$-categories}}

\nc{\inftyOneCat}{\term{$(\infty,1)$-category}}
\nc{\inftyOneCats}{\term{$(\infty,1)$-categories}}

\nc{\inftyGrpd}{\term{$\infty$-groupoid}}
\nc{\inftyGrpds}{\term{$\infty$-groupoids}}

\nc{\inftyTop}{\term{$\infty$-topos}}
\nc{\inftyTops}{\term{$\infty$-toposes}}

\nc{\inftyTwoCat}{\term{$(\infty,2)$-category}}
\nc{\inftyTwoCats}{\term{$(\infty,2)$-categories}}

\title{Hermitian K-theory via oriented Gorenstein algebras}

\author[M. Hoyois]{Marc Hoyois}
\address{Fakultät für Mathematik\\
Universität Regensburg\\
Universitätsstr. 31\\
93040 Regensburg\\
Germany}
\email{\href{mailto:marc.hoyois@ur.de}{marc.hoyois@ur.de}}
\urladdr{\url{http://www.mathematik.ur.de/hoyois/}}

\author[J. Jelisiejew]{Joachim Jelisiejew}
\address{Faculty of Mathematics, Informatics
and Mechanics\\
University of Warsaw\\
Banacha 2\\
02-097 Warsaw\\
Poland}
\email{\href{mailto:jjelisiejew@mimuw.edu.pl}{jjelisiejew@mimuw.edu.pl}}
\urladdr{\url{https://www.mimuw.edu.pl/~jjelisiejew/}}

\author[D. Nardin]{Denis Nardin}
\address{Fakultät für Mathematik\\
Universität Regensburg\\
Universitätsstr. 31\\
93040 Regensburg\\
Germany}
\email{\href{mailto:denis.nardin@ur.de}{denis.nardin@ur.de}}
\urladdr{\url{https://homepages.uni-regensburg.de/~nad22969/}}

\author[M. Yakerson]{Maria Yakerson}
\address{Institute for Mathematical Research (FIM)\\
ETH Z\"urich \\
R\"amistr. 101\\  
8092 Z\"urich\\
Switzerland}
\email{\href{mailto:maria.yakerson@math.ethz.ch}{maria.yakerson@math.ethz.ch}}
\urladdr{\url{https://www.muramatik.com}}

\thanks{M.H., D.N., and M.Y.\ were partially supported by SFB 1085 ``Higher invariants''}
\thanks{J.J. was supported by NCN grant 2017/26/D/ST1/00913 and by the START
fellowship of the Foundation for Polish Science}

\date{\today}

\begin{document}

\begin{abstract}
We show that the hermitian K-theory space of a commutative ring $R$ can be identified, up to $\A^1$-homotopy, with the group completion of the groupoid of oriented finite Gorenstein $R$-algebras, i.e., finite locally free $R$-algebras with trivialized dualizing sheaf.
We deduce that hermitian K-theory is universal among generalized motivic cohomology theories with transfers along oriented finite Gorenstein morphisms. As an application, we obtain a Hilbert scheme model for hermitian K-theory as a motivic space. We also give an application to computational complexity: we prove that $1$-generic minimal border rank tensors degenerate to the big Coppersmith--Winograd tensor.
\end{abstract}

\maketitle

\parskip 0.2cm

\parskip 0pt
\setcounter{tocdepth}{1}
\tableofcontents

\parskip 0.2cm
\vspace{-2em}

\section{Introduction}

A recent development in motivic homotopy theory is the theory of framed transfers, introduced in \cite{voevodsky2001notes} and studied further in \cite{garkusha2014framed}, \cite{deloop1} and other works. The fundamental results of this theory are as follows. First, the \emph{reconstruction theorem} \cite[Theorem~3.5.12]{deloop1} states that every generalized motivic cohomology theory has a unique structure of (coherent) framed transfers, i.e., has covariance along finite syntomic maps of schemes equipped with a trivialization of the cotangent complex. Second, the \emph{motivic recognition principle} \cite[Theorem~3.5.14]{deloop1} shows that every $\P^1$-connective (also called ``very effective") motivic spectrum is a suspension spectrum on its infinite $\P^1$-loop space, when framed transfers on the space are taken into account. With these tools in hand, several important motivic spectra can be described geometrically as framed suspension spectra of certain moduli stacks of schemes. For example, the algebraic cobordism spectrum is the framed suspension spectrum of the stack of finite syntomic schemes \cite[Theorem~3.4.1]{deloop3}, and the effective algebraic K-theory spectrum is the framed suspension spectrum of the stack of finite flat schemes \cite[Theorem~5.4]{robbery}. 

From the point of view of moduli of schemes, there are three natural conditions of interest: Cohen--Macaulay, Gorenstein, and local complete intersection,
see~\cite[Chapters~8--9]{HarDeform}. Being Cohen--Macaulay is automatic for finite
schemes, so by the discussion above, the first and third conditions relate to algebraic K-theory and
algebraic cobordism, respectively.

In this article, we connect the second condition, being Gorenstein, to another important cohomology theory, namely, hermitian K-theory. We show that, over a field of characteristic not $2$, hermitian K-theory is the framed suspension spectrum of the stack of finite Gorenstein algebras equipped with an \emph{orientation}, meaning a trivialization of the dualizing line bundle (Theorem~\ref{thm:KQ-kq-FGor}):
\[\kq \simeq  \Sigma^\infty_\fr \FGor^\o.\]
 This identification allows us to complete the description of all well-known generalized motivic cohomology theories as motivic spectra with framed transfers, which we summarize in \sectsign\ref{ssec:big} below (see also Theorem~\ref{thm:all mot spectra}). 

The infinite $\P^1$-loop space of $\kq$ is motivically equivalent to the group completion of the stack $\Vect^\sym$ of vector bundles with a non-degenerate symmetric bilinear form.
Using the motivic recognition principle, the equivalence $\kq \simeq \Sigma^\infty_\fr \FGor^\o$ follows from the fact that the forgetful map of commutative monoids $\FGor^\o \to \Vect^\sym$ induces an $\A^1$-equivalence after group completion (Theorem~\ref{thm:main-gp}); the $\A^1$-inverse map is given by sending a symmetric vector bundle $V$ to $[V \oplus \Oo[x]/x^2] - [\Oo[x]/x^2] $. To prove this result, we employ several explicit $\A^1$-homotopies. A key ingredient is the intermediate notion of an \emph{isotropically augmented} oriented Gorenstein algebra, which is a condition that allows to split off a copy of the dual numbers in a controlled way.

We would like to emphasize that, although Theorem~\ref{thm:main-gp} is analogous to~\cite[Theorem~3.1]{robbery}, the proof in \emph{loc.\ cit.}\ cannot be applied in the hermitian context. Indeed, while the square-zero extension functor $\Vect_d \to \FFlat_{d+1}$ is an $\A^1$-equivalence, the analogous functor $\Vect^\sym_d \to \FGor^\o_{d+2}$, which  adds a copy of the dual numbers, is not an $\A^1$-equivalence; we show this explicitly for the base field $\bR$ in Proposition~\ref{prop: square zero ext over R}. Instead, $\Vect^\sym_d$ is $\A^1$-equivalent to the substack of \emph{isotropic} oriented Gorenstein algebras $\FGor^{\o,0}_{d+2}$, defined by the condition that the orientation vanishes at the unit.

As a corollary of our results, we get a Hilbert scheme model for the Grothendieck–Witt space $\GWspace$ (Corollary~\ref{cor:Hilb-FGor}): we show that it is motivically equivalent to $\Z \times \Hilb_\infty^{\Gor,\o}(\A^\infty)$, where $\Hilb_d^{\Gor,\o}(\A^n)$ is the Hilbert scheme of degree $d$ finite Gorenstein subschemes of $\A^n$ equipped with an orientation. 
We also give a new proof of the motivic equivalence $\GWspace \simeq \Z\times \GrO_\infty$ due to Schlichting and Tripathi \cite{SchlichtingTripathi}, where $\GrO_d$ is the orthogonal Grassmannian (see Corollary~\ref{cor: Vect^bil and GrO}).
A notable difference between these two models is that, when 2 is not invertible, the oriented Hilbert scheme gives the ``genuine symmetric'' variant of Grothendieck--Witt theory defined in \cite{HermitianI,HermitianII}, whereas the orthogonal Grassmannian corresponds to the ``genuine even'' variant.

\begin{rem}
	None of the techniques in this paper require $2$ to be invertible. This assumption only appears in statements that involve the motivic spectra $\KQ$ or $\kq$, which so far have only been defined over $\Z[\tfrac 12]$-schemes (see however Remark~\ref{rem:2}).
\end{rem}

Our arguments have an unexpected direct application to complexity theory. One of the central problems there is bounding the asymptotic complexity of matrix multiplication~\cite{burgisser_clausen_shokrollahi, Landsberg__complexity_book}. The current algorithm giving the upper bound (called the \emph{laser method}) rests on the existence of certain tensors of minimal border rank, most notably the big Coppersmith--Winograd tensor $\mrm{CW}_q$. In Corollary~\ref{cor:degTensors} we show that \emph{every} $1$-generic tensor of minimal border rank degenerates to $\mrm{CW}_q$, thus the latter is the most degenerate one; it has the smallest value, which is very surprising given that the whole algorithm depends critically on the value of $\mrm{CW}_q$. Equally surprising is the fact that this result is seemingly not known to researchers in tensor world, even though it is much easier to obtain than our main results, since the degeneration need not be canonical.

Oriented Gorenstein algebras are also called commutative Frobenius algebras and appear in a different context: they classify 2-dimensional topological field theories~\cite[Section~1.1]{LurieTFT}. 

\subsection{Framed models for motivic spectra}\label{ssec:big} 
We give here an overview of the known framed models for motivic spectra. We work over a base field $k$ for simplicity and refer to Section~\ref{sec:framed-models} for more detailed statements and references.
Consider the following presheaves of $\infty$-groupoids on the category of smooth $k$-schemes:
\begin{itemize}
	\item $\uZ$ is the constant sheaf with fiber $\Z$;
	\item $\uGW$ is the sheaf of unramified Grothendieck--Witt groups;
	\item $\Vect$ is the groupoid of vector bundles;
	\item $\Vect^\sym$ is the groupoid of non-degenerate symmetric bilinear forms;
	\item $\FFlat$ is the groupoid of finite flat schemes;
	\item $\FGor^\o$ is the groupoid of oriented finite Gorenstein schemes;
	\item $\FSyn$ is the groupoid of finite syntomic schemes;
	\item $\FSyn^\o$ is the groupoid of oriented finite syntomic schemes;
	\item $\FSyn^\fr$ is the $\infty$-groupoid of framed finite syntomic schemes \cite[3.5.17]{deloop1}.
\end{itemize}
Then the forgetful maps
\[\begin{tikzcd}
     & \FSyn  \ar{r} & \FFlat  \ar{r} & \Vect \ar{r} & \uZ \\
   \FSyn^\fr \ar{r} & \FSyn^\o \ar{u} \ar{r} & \FGor^\o \ar{r} \ar{u}& \Vect^\sym \ar{u}\ar{r} & \uGW\ar{u} 
\end{tikzcd}\]
induce, upon taking framed suspension spectra, the canonical morphisms of motivic $\Einfty$-ring spectra over $k$ (assuming $\operatorname{char} k\neq 2$ for $\kq$):
\[\begin{tikzcd}
     & \MGL  \ar{r} & \kgl  \ar["\simeq"]{r} & \kgl \ar{r} & \hz \\
   \MonUnit \ar{r} & \MSL \ar{u} \ar{r} & \kq \ar["\simeq"]{r} \ar{u}& \kq \ar{u}\ar{r} & \hzmw \ar{u} \rlap .
\end{tikzcd}\]
Here, $\MonUnit$ is the motivic sphere spectrum, $\MGL$ (resp.\ $\MSL$) is the algebraic (resp.\ special linear) cobordism spectrum, $\kgl$ (resp.\ $\kq$) is the very effective algebraic (resp.\ hermitian) K-theory spectrum, and $\hz$ (resp.\ $\hzmw$) is the motivic cohomology (resp.\ Milnor–Witt motivic cohomology) spectrum.

One application of these geometric models is that they allow us to describe modules over these motivic ring spectra in terms of the corresponding transfers. For example, we prove in Theorem~\ref{thm:kq-modules} that modules over hermitian K-theory are equivalent (at least in characteristic 0) to generalized motivic cohomology theories with coherent transfers along oriented finite Gorenstein maps. The fact that hermitian K-theory has such transfers is well-known (see for example \cite[\sectsign 0]{GilleTransfers}), and this result characterizes hermitian K-theory as the universal such cohomology theory. 

\subsection*{Acknowledgments} We are thankful to Tom Bachmann, Joseph M. Landsberg, Rahul Pandharipande and Burt Totaro for helpful discussions. We would like to thank SFB 1085 ``Higher invariants'' and Regensburg University for its hospitality. Yakerson was supported by a Hermann-Weyl-Instructorship and is grateful to the Institute of Mathematical Research (FIM) and to ETH Z\"urich for providing perfect working conditions.

\section{Oriented Gorenstein algebras}
In this section we introduce several types of algebras and their properties, which will later be used in the proof of the main theorem. All rings and algebras are assumed commutative.

Let $R$ be a base ring. An $R$-algebra $A$ is \emph{finite locally free} if it is finite, flat, and of finite presentation, or equivalently if $A$ is a locally free $R$-module of finite rank.
A finite locally free $R$-algebra $A$ is called a \emph{Gorenstein} $R$-algebra if its dualizing module $\omega_{A/R} = \Hom_R(A, R)$ is an invertible $A$-module~\cite[Proposition~21.5]{Eisenbud}. 
We denote by $\FGor(R)$ the groupoid of Gorenstein $R$-algebras.
We emphasize that for us a Gorenstein algebra is by definition finite locally free; in this article we do not consider more general Gorenstein algebras, so this usage will not be ambiguous.

\begin{defn}\label{dfn:oriented-Gorenstein-algebra}
 Let $A$ be a Gorenstein $R$-algebra. An \emph{orientation} of $A$ is an element $\varphi \in \omega_{A/R}$ that trivializes the dualizing module of $A$. Equivalently, $\varphi \colon A\to R$ is an $R$-linear homomorphism such that the bilinear form $B_\varphi(x,y)=\varphi(xy)$ on $A$ is non-degenerate (i.e., $B_\varphi$ induces an isomorphism $A \simeq A^\vee$). An \emph{oriented} Gorenstein $R$-algebra is a pair $(A, \varphi)$ where $\varphi$ is an orientation of the Gorenstein $R$-algebra $A$. 
We denote by $\FGor^\o(R)$ the groupoid of oriented Gorenstein $R$-algebras. 
 \end{defn}

\begin{rem}
An oriented Gorenstein $R$-algebra is the same thing as a commutative Frobenius algebra in the category of $R$-modules, in the sense of \cite[Definition 4.6.5.1]{HA}.
\end{rem}

\begin{rem}
Let $R$ be a field or, more generally, a semilocal ring. Then every Gorenstein $R$-algebra can be oriented, but the choice of an orientation is usually far from unique. For a Gorenstein algebra $A$ over a more general ring $R$ an orientation may not exist, since $\omega_{A/R}$ may be a non-trivial line bundle.
\end{rem}

\begin{ex} \label{ex:kx/x2}
 The algebra $R[x]/x^2$ can be equipped with the orientation $\varphi_0(r+sx)=s$, turning it into an oriented Gorenstein $R$-algebra.
 The corresponding bilinear form $B_{\varphi_0}$ has matrix $\left(\begin{smallmatrix}0 & 1\\ 1& 0\end{smallmatrix}\right)$ in the basis $(1,x)$.
 This algebra is therefore a refinement of the hyperbolic plane. As such, it will play a crucial role in Section~\ref{sec:gp}.
\end{ex}

\begin{ex}
More generally, consider the Gorenstein $R$-algebra $A = R[x]/x^{n+1}$. One can check that a functional $\varphi\colon A\to R$ is an orientation if and only if $\varphi(x^n) \in R^\times$. 
\end{ex}

\begin{rem}\label{rem:FGoror}
	The presheaves of groupoids $\FGor$ and $\FGor^\o$ are algebraic stacks.
	To see this, consider the universal finite flat family $p\colon U \to \FFlat$.
	By \cite[Tag 05P8]{stacks}, the locus $U_1\subset U$ where the sheaf $\omega_{U/\FFlat}$ is locally free of rank $1$ is open. By definition, $\FGor$ is the Weil restriction of $U_1$ along $p$, hence it is an open substack of $\FFlat$ \cite[\sectsign 7.6 Proposition 2(i)]{NeronModels}.
	The forgetful map $\FGor^\o\to \FGor$ is similarly the Weil restriction of a $\G_m$-torsor over $U\times_\FFlat\FGor$, hence it is affine. Alternatively, we can consider the vector bundle $E=\Spec\Sym(p_*\sO_U)$ over $\FFlat$. As a stack, $E$ classifies pairs $(A,\phi)$ where $A$ is a finite locally free $R$-algebra and $\phi\in\omega_{A/R}$. Thus, $\FGor^\o$ is the open substack of $E$ where the determinant of $B_\phi$ does not vanish. This shows moreover that the forgetful map $\FGor^\o\to \FFlat$ is affine, since the complement of $\FGor^\o$ in $E$ is cut out by a single equation (locally on $\FFlat$).
\end{rem}

\begin{lem}\label{lem:perpendiculars_to_ideals_are_intrinsic}
    Let $R$ be a ring, $A$ a Gorenstein $R$-algebra and $I\subset A$ an
    ideal. For every orientation $\varphi$ of $A$, we have $I^{\perp} = \Ann(I)$, where the orthogonal is taken with respect to $B_{\varphi}$.
\end{lem}

\begin{proof}
    Take $a\in A$. If $aI = 0$ we have $B_{\varphi}(a, i) = \varphi(ai) =
    0$ for every $i\in I$, hence $a\in I^{\perp}$. Conversely, if $aI\neq 0$ then, since $B_{\varphi}$ is
    non-degenerate, there exist $i\in I$ and $a'\in A$ such that $B_{\varphi}(a', ai)\neq 0$. This means that 
    $B_{\varphi}(a,a'i)=\varphi(aa'i)=B_{\varphi}(a',ai)\neq 0$ so $a\not\in I^\perp$.
\end{proof}

Recall that for an $R$-algebra $A$ an \emph{augmentation} is an $R$-algebra map $e \colon A \to R$. When $(A,\varphi)$ is an oriented Gorenstein $R$-algebra, we denote by $e^*\colon R\to A$ the $R$-linear adjoint map, i.e., the map such that $B_\varphi(e^*\lambda,y)=\lambda e(y)$.

\begin{defn}\label{def:isotropic-aug}
	Let $(A,\phi)$ be an oriented Gorenstein $R$-algebra. An augmentation $e\colon A\to R$ is called \emph{isotropic} if $e^*(1)$ is isotropic, i.e., if $B_\phi(e^*(1),e^*(1))=0$.
	An \emph{isotropically augmented} oriented Gorenstein $R$-algebra is a triple $(A, \varphi, e)$ where $e$ is an isotropic augmentation of the oriented Gorenstein $R$-algebra $(A, \varphi)$. We denote by $\FGor^{\o,+}(R)$ the groupoid of isotropically augmented oriented Gorenstein $R$-algebras.
\end{defn}

\begin{prop}\label{prop:loosingorientation}
    Let $(A,\phi)$ be an oriented Gorenstein $R$-algebra and $e\colon A\to R$ an augmentation. Then
	 $e$ is isotropic if and only if $\Ann(\ker e)\subset \ker e$.
\end{prop}

\begin{proof}
	 By definition, $e$ is isotropic if and only if the image of $e^*\colon R\to A$ is contained in $\ker e$.
	 By dualizing the short exact sequence of $R$-modules \[0\to\ker e\to A\stackrel e\to R\to 0,\] we see that $\Ima e^*=(\ker e)^\perp$.
	 By Lemma~\ref{lem:perpendiculars_to_ideals_are_intrinsic}, we get $\Ima e^*=\Ann(\ker e)$, whence the result.
\end{proof}

\begin{ex}\label{ex:isotropic}
Let $A$ be a finite algebra over a field $k$, so that $A$ is a finite product of local algebras
    $A_\mathfrak{m}$. A choice of augmentation is then tantamount to a choice of $\mathfrak m$ such that $k\to \kappa(\mathfrak{m})$ is an isomorphism. This augmentation is isotropic (for any choice of orientation) if and only if $A_{\mathfrak{m}} \not\simeq k$.
\end{ex}

\begin{rem}
	Suppose $R$ is a \emph{reduced} ring.
	If $(A,\phi)$ is an oriented Gorenstein $R$-algebra with augmentation $e$, then $e$ is isotropic if and only it is isotropic after base change to the residue fields of $R$ (where isotropicity is a simple geometric condition, see Example~\ref{ex:isotropic}). Indeed, if $R$ is reduced then the map $R\to \prod_{\mathfrak p} \kappa(\mathfrak p)$ is injective, where $\mathfrak p$ ranges over the minimal primes of $R$.
\end{rem}

\begin{rem}
	The condition $\Ann(\ker e)\subset \ker e$ in Proposition~\ref{prop:loosingorientation} is independent of the orientation $\phi$. Thus one could define isotropically augmented finite locally free algebras, but we do not know of an application for such a notion.
\end{rem}

\begin{defn}
Let $(A,\varphi,e) \in \FGor^{\o,+}(R)$. The element $e^\ast(1) \in A$ is called the \emph{local socle generator} of $(A,\varphi,e)$.
\end{defn}

By definition of $e^*$, the local socle generator $x=e^*(1)$ has the property that $B_\varphi(x,y)=e(y)$ for each $y\in A$. In particular, by the isotropy condition, we have $e(x) =B_\varphi(x,x) = 0$.
The name reflects the fact that $x$ generates the \emph{socle} (the annihilator of the maximal ideal) of the localization of $A$ at $\ker e$, see Lemma~\ref{lem:local_socle}(1) below.

\begin{ex}\label{ex:dualnumbers}
    The oriented Gorenstein $R$-algebra $(R[x]/x^2, \varphi_0)$ from Example~\ref{ex:kx/x2} has a unique isotropic augmentation $e_0$ given by $e_0(r+sx) = r$. Its local socle generator is $x$.
\end{ex}

\begin{ex}\label{ex:Gorenstein-algebras-from-bilinear-forms}
 Let $(V,B)$ be a non-degenerate symmetric bilinear form over a ring $R$. Then we can construct an algebra $A:=R[x]/x^2\oplus V$ with multiplication
 \[(r+sx,v)\cdot(r'+s'x,v'):=(rr'+(rs'+r's+B(v,v'))x, r'v+rv')\,.\]
 If we let $\varphi(r+sx,v)=s$ and $e(r+sx,v)=r$, then the triple $(A,\varphi,e)$ is an isotropically augmented oriented Gorenstein $R$-algebra, with local socle generator $x$. Note that the bilinear form $B_\phi$ is the direct sum of the hyperbolic form on $R[x]/x^2$ (with respect to the basis $(1,x)$) and the original form $B$ on $V$.
\end{ex}

\begin{lem}\label{lem:local_socle}
 Let $x$ be the local socle generator of $(A,\varphi,e)\in  \FGor^{\o,+}(R)$. Then
 \begin{enumerate}
     \item $x$ spans the $R$-module $\Ann(\ker e)$ (in particular, $R x \subset A$ is an
         ideal);
     \item $x^2=0$;
     \item $\varphi(x)=1$.
 \end{enumerate}
 In particular, there is a canonical copy of $R[x]/x^2$ inside $A$, with an induced orientation given by $\varphi(r+sx)=\varphi(1)r+s$.
\end{lem}

\begin{proof}
	By definition, $x$ spans $\Ima e^*=(\ker e)^\perp$ (see the proof of Proposition~\ref{prop:loosingorientation}), which is $\Ann(\ker e)$ by Lemma~\ref{lem:perpendiculars_to_ideals_are_intrinsic}.
	By Proposition~\ref{prop:loosingorientation}, $\Ann(\ker e)$ is a square-zero ideal, hence $x^2=0$. Finally, $\phi(x)=B_\phi(x,1)=e(1)=1$, since $e$ is a ring homomorphism.
\end{proof}

\begin{defn}
A \emph{non-unital oriented Gorenstein} $R$-algebra is a pair $(V, B)$ where $V$ is a finite locally free non-unital (i.e., not necessariy unital) $R$-algebra and $B$ is a non-degenerate symmetric bilinear form on $V$ such that $B(xy,z)=B(x,yz)$ for all $x,y,z \in V$. We denote by $\FGor^{\nonu}(R)$ the groupoid of non-unital oriented Gorenstein $R$-algebras.
\end{defn}

\begin{rem}
If $(V,B)$ is a non-unital oriented Gorenstein $R$-algebra such that $V$ is unital, we can define an $R$-linear map $\varphi\colon V\to R$ by $\varphi(y) = B(y, 1)$. By construction, we then have $B_\varphi = B$, so $(V, \varphi)$ is an ordinary oriented Gorenstein $R$-algebra.
\end{rem}

If we view the moduli stack $\FGor^\o$ of oriented Gorenstein algebras as the hermitian counterpart of the moduli stack $\FFlat$ of finite locally free algebras, then $\FGor^{\o,+}$ is the hermitian counterpart of the moduli stack $\FFlat^\mrk$ of augmented finite locally free algebras, and $\FGor^\nonu$ that of the moduli stack $\FFlat^\nonu$ of non-unital finite locally free algebras. In the finite locally free case, the augmentation ideal defines an equivalence of presheaves of groupoids $\Aug\colon\FFlat^\mrk\to \FFlat^\nonu$, whose inverse is given by unitalization. We now investigate the hermitian analogue of this equivalence, which is slightly more complicated.

\begin{constr}\label{constr:non-unitalization}
    If $(A, \varphi, e) \in \FGor^{\o,+}(R)$ has local socle generator $x$, there is a direct sum decomposition 
	 \[
	 \ker e=Rx\oplus (R[x]/x^2)^\perp.
	 \]
	 Indeed, since $B_\phi(x,y)=e(y)$ we have $\ker e=(Rx)^\perp$, which contains the right-hand side as $x$ is isotropic. Conversely, if $y\in \ker e$, then $y-B_\phi(y,1)x$ is orthogonal to $1$ by Lemma~\ref{lem:local_socle}(3) and to $x$, hence it is orthogonal to $R[x]/x^2$.
	 In particular, $(R[x]/x^2)^\perp$ can be identified with the quotient $\ker(e)/Rx$.
	Since $Rx$ is an ideal in $A$ by Lemma~\ref{lem:local_socle}(1), $((R[x]/x^2)^\perp, B_\varphi)$ is a non-unital oriented Gorenstein $R$-algebra with multiplication given by the multiplication in $A/Rx$ (or equivalently, by the multiplication in $A$ followed by the orthogonal projection).
\end{constr}

\begin{prop}\label{prop:nonunital=augmented}
    The map
    \[\Aug\colon\FGor^{\o,+}\to \A^1\times\FGor^{\nonu}\]
    sending $(A,\varphi,e)$ to the pair $(\varphi(1),((R[x]/x^2)^\perp, B_\varphi))$ is an equivalence of presheaves of groupoids.
\end{prop}

\begin{proof}
    We construct the inverse map $\Uni\colon \A^1\times\FGor^{\nonu} \to \FGor^{\o,+}$ by associating to $\lambda\in R$ and $(V, B) \in \FGor^{\nonu}(R)$ the $R$-module $R[x]/x^2\oplus V$ with the multiplication
    \begin{equation}\label{eq:multiplication-in-augmented-algebras}
        (r+sx,v)(r'+s'x,v')=(rr'+(sr'+s'r+B(v,v'))x, r'v+rv'+vv')\,,
    \end{equation}
    and augmentation and orientation given by
    \[e(r+sx,v)=r,\qquad \varphi(r+sx,v)=r\lambda+s\,.\] 
	 It is straightforward to check that this multiplication is associative. Clearly, this assignment is functorial, and $\Aug \circ \Uni \simeq \id$, so it remains to show that $\Uni \circ \Aug\simeq\id$. 
    
    Let $(A,\varphi,e) \in \FGor^{\o,+}(R)$ with the local socle generator $x$,  and let $(V, B)$ be the orthogonal complement of $R[x]/x^2$ in $A$ equipped with the non-unital Gorenstein algebra structure of Construction~\ref{constr:non-unitalization}. Then there is a canonical orthogonal decomposition $A\simeq R[x]/x^2\oplus V.$
    Since the kernel of the augmentation is exactly $(R\cdot x)^\perp$ (by definition of $x$), the augmentation is given by projecting onto the first summand and setting $x=0$. Moreover, $V$ is contained in the orthogonal complement of $R\cdot 1$ and so
    \[\varphi(r+sx,v)=\varphi(r+sx,0)=r\lambda+s\,.\]
    Finally, we need to check that the multiplication is given by \eqref{eq:multiplication-in-augmented-algebras}. Since $x$ annihilates $V$, the only question is what is the product of two elements of the form $(0,v)$ and $(0,v')$. The second component is, by definition, the product in $V$. Since the product still needs to be in $\ker e$, it follows that
    \[(0,v)(0,v')=(bx,vv')\]
    for some $b\in R$. But then
    \[b=\varphi(bx,vv')=\varphi\left((0,v)(0,v')\right)=B_\varphi\left((0,v),(0,v')\right)=B(v,v')\]
    as required.
\end{proof}

The following table summarizes the various presheaves of groupoids of finite flat algebras considered in this paper:
\begin{center}
  \begin{tabular}{l l}
  name & description\\
    \toprule
        $\FFlat$ & finite locally free algebras\\
        $\FFlat^\mrk$ & augmented finite locally free algebras\\
		$\FFlat^\nonu$ & non-unital finite locally free algebras\\
         $\FGor$ & (finite) Gorenstein algebras\\
         $\FGor^\o$ & oriented Gorenstein algebras\\
        $\FGor^{\o,+}$ & isotropically augmented oriented Gorenstein algebras\\
    $\FGor^\nonu$ & non-unital oriented Gorenstein algebras\\
$\FGor^{\o,0}$ & isotropic oriented Gorenstein algebras (see Section~\ref{sec:isotropic})
\end{tabular}
\end{center}

All of these groupoids of $R$-algebras extend by descent to groupoids of quasi-coherent $\Oo_S$-algebras for an arbitrary scheme $S$. 
For example, $\FGor^\o(S)$ is the groupoid of finite locally free quasi-coherent $\sO_S$-algebras $\mathcal{A}$ equipped with an $\Oo_S$-linear map $\varphi\colon \mathcal{A}\to \Oo_S$ such that the induced bilinear form $B_\varphi\colon \sA\times\sA\to \sO_S$ is non-degenerate.
We will also identify quasi-coherent $\Oo_S$-algebras with $S$-schemes that are affine over $S$ (in particular in the discussion of Hilbert schemes). Thus, $\FGor^\o(S)$ can be identified with the groupoid of finite locally free $S$-schemes $p\colon Z\to S$ together with a suitable map $\varphi\colon p_*\Oo_Z\to \Oo_S$. In this situation, we will often abuse notation and regard $\sO_Z$ as a quasi-coherent $\sO_S$-algebra.

In order to construct some $\A^1$-homotopies in the proof of our main result, we will need a way to glue together objects of $\FGor^{\o,+}(S)$ along the basepoint. 
Given $(Z_1, \varphi_1, e_1), (Z_2, \varphi_2, e_2) \in \FGor^{\o,+}(S)$, 
the disjoint union $Z_1\sqcup Z_2$ inherits an orientation 
\[
\phi\colon \sO_{Z_1\sqcup Z_2}=\sO_{Z_1}\times \sO_{Z_2}\to \sO_S,\quad \phi(a_1,a_2)=\phi_1(a_1)+\phi_2(a_2),
\]
which makes $\sO_{Z_1\sqcup Z_2}$ into the orthogonal sum of $\sO_{Z_1}$ and $\sO_{Z_2}$.
Gluing $Z_1$ and $Z_2$ along the basepoint means passing to the subalgebra $\sO_{Z_1\sqcup_SZ_2}=\sO_{Z_1}\times_{\sO_S} \sO_{Z_2}\subset \sO_{Z_1}\times \sO_{Z_2}$. However, the restriction of $\phi$ to $\sO_{Z_1\sqcup_SZ_2}$ is no longer an orientation:
if $x_i$ is the local socle generator of $(Z_i,\varphi_i,e_i)$, then $(x_1,-x_2)$ belongs to the radical of $B_\phi$ on $\sO_{Z_1\sqcup_SZ_2}$. It turns out that this is the only obstruction and that we obtain a well-defined orientation on the vanishing locus of $(x_1,-x_2)$:

\begin{prop}\label{prop:connectedSum}
	Let $(Z_1, \varphi_1, e_1), (Z_2, \varphi_2, e_2) \in \FGor^{\o,+}(S)$ and let $x_i\in\sO(Z_i)$ be the corresponding local socle generators.
	Let $Z_{12} \subset Z_1\sqcup_{S} Z_2$ be the closed subscheme given by the equation $(x_1,-x_2)$, let $e=e_1=e_2\colon S\to Z_{12}$, and let $\phi=\phi_1+\phi_2\colon \sO_{Z_{12}}\subset \sO_{Z_1}\oplus\sO_{Z_2}\to \sO_S$.
	Then $Z_{12}$ is a finite Gorenstein $S$-scheme of degree $\deg(Z_1) + \deg(Z_2) - 2$, with orientation $\phi$ and isotropic augmentation $e$.
\end{prop}

\begin{proof}
    First we need to prove that $Z_{12}$ is finite locally free over $S$. Since this property is local on the base we can assume $S=\Spec(R)$ is affine. Then we can write $Z_i=\Spec(A_i)$ and $Z_{12}=\Spec (A_1\times_R A_2)/(x_1,-x_2)$. We know that $A_1\times_R A_2$ is finite locally free over $R$ by \cite[Lemma~3.6]{robbery}. The inclusion of the ideal spanned by $x_i$ is a split inclusion $R\to A_i$, with the splitting given by $\varphi_i$, so the inclusion of the ideal spanned by $(x_1,-x_2)$ is also a split inclusion $R\to A_1\times_R A_2$ (split, say, by $\varphi_1\circ\mathrm{pr}_1$). In particular, the quotient $(A_1\times_R A_2)/(x_1,-x_2)$ is also finite locally free, and of the desired degree.
    
	 To show that $\phi$ is an orientation, we can again work in the affine case. We have to show that the radical of the form $B_\phi$ on $A_1\times_RA_2$ is precisely the ideal $R(x_1,-x_2)$. 
	 On the one hand,
	 \[
	 B_\phi((x_1,-x_2),(y_1,y_2)) = e_1(y_1)-e_2(y_2) = 0
	 \]
	 for all $(y_1,y_2)\in A_1\times_RA_2$, so $(x_1,-x_2)$ is in the radical. By Lemma~\ref{lem:local_socle}, we can write $A_i=R[x_i]/x_i^2\oplus V_i$ with $V_i=(R[x_i]/x_i^2)^\perp$, whence
	 \[
	 A_1\times_RA_2 \simeq R[x_1,x_2]/(x_1,x_2)^2 \oplus V_1\oplus V_2.
	 \]
	 The restriction of $B_\phi$ to $R[x_1,x_2]/(x_1,x_2)^2$ is given by a matrix $\left(\begin{smallmatrix}b & 1 \\ 1 & 0\end{smallmatrix}\right)\oplus 0$ in the basis $(1,x_1,x_1-x_2)$, hence has radical $R(x_1-x_2)$. The latter contains the radical of $B_\phi$ since the restriction of $B_\phi$ to $V_1\oplus V_2$ is non-degenerate.
	 Finally, the augmentation $e$ is isotropic since $e^*(1)=(x_1,0)$ and hence $e(e^*(1))=e_2(0)=0$.
\end{proof}

\begin{defn}\label{def:connected sum}
	In the setting of Proposition~\ref{prop:connectedSum}, $(Z_{12},\phi,e)\in\FGor^{\o,+}(S)$ is called the \emph{connected sum} of $(Z_1,\phi_1,e_1)$ and $(Z_2,\phi_2,e_2)$.
\end{defn}

\begin{rem}
	In the affine case, Definition~\ref{def:connected sum} is a special case of the more general notion of connected sum of rings studied in \cite[Section 2]{Ananthnarayan_Connected_sums}. In particular, the fact that $Z_{12}$ is Gorenstein is also a consequence of \cite[Theorem~2.8]{Ananthnarayan_Connected_sums}.
\end{rem}

\begin{rem}
It is easy to show that the connected sum gives a commutative monoid structure on the stack $\FGor^{\o,+}$. Moreover, if $\FFlat^\mrk$ denotes the moduli stack of pointed finite locally free schemes with the commutative monoid structure given by the wedge sum, the morphism $\Hyp\colon \FFlat^\mrk\to \FGor^{\o,+}$ from Remark~\ref{rem:functor_hyp} is a morphism of commutative monoids. We shall not need these facts in the sequel.
\end{rem}

\section{$\A^1$-equivalence between the group completions of the stacks $\FGor^\ori$ and $\Vect^\sym$}
\label{sec:gp}

We denote by $\Vect^\sym$ the stack of finite locally free modules equipped with a non-degenerate symmetric bilinear form.
The direct sum and tensor product define an $\Einfty$-semiring structure on $\Vect^\sym$. Similarly, the disjoint union and cartesian product of schemes define an $\Einfty$-semiring structure on $ \FGor^\o$, and the forgetful map $\eta\colon\FGor^\o \to \Vect^\sym$, sending $(A,\varphi)$ to $(A,B_\varphi)$, is a morphism of $\Einfty$-semirings. We will describe these constructions more explicitly in Section~\ref{sec:kq}. The main result of this section is the following theorem.

\begin{thm}\label{thm:main-gp}
	The map $\eta^\gp\colon \FGor^{\o,\gp} \to \Vect^{\sym,\gp}$ is an $\A^1$-equivalence of presheaves on the category of schemes, where $\gp$ stands for objectwise group completion.
\end{thm}

\begin{rem}\label{rem:functor_hyp}
Theorem~\ref{thm:main-gp} is analogous to~\cite[Theorem~2.1]{robbery}, which
states that the forgetful map from the stack of finite locally free schemes
$\FFlat$ to the stack of vector bundles $\Vect$ becomes an $\A^1$-equivalence
after group completion. The connection can be expressed precisely in the following way. 
There is a commutative diagram of hyperbolic and forgetful functors
\[\begin{tikzcd}
     \FFlat\ar[d,swap,"\Hyp"] \ar[r, "\eta"] & \Vect\ar[d,"\Hyp"]\\
    \FGor^{\ori}\ar[r, "\eta"] \ar[d,swap,"\mathrm{forget}"] & \Vect^{\sym} \ar[d,"\mathrm{forget}"] \\
	 \FFlat \ar[r, "\eta"] & \Vect\rlap.
\end{tikzcd}\]
Here, the functor $\Hyp\colon \Vect\to\Vect^\sym$ sends $V$ to $\left(V\oplus V^\vee, \left(\begin{smallmatrix}
    0 & I\\ I & 0\end{smallmatrix}\right)\right)$, and the functor $\Hyp\colon \FFlat\to \FGor^\o$ sends a finite locally free $R$-algebra $A$ to the square-zero extension $A\oplus\omega_{A/R}$ with the orientation $\varphi(a,f)=f(a)$.
\end{rem}

To proceed, we define the following stabilization maps:
\begin{align*}
\tau \colon \Vect^\sym \to  \Vect^\sym, &\quad (V,B) \mapsto (V\oplus R^2, B \oplus B_{\Hyp}), \\
\sigma \colon  \FGor^{\o} \to  \FGor^{\o}, &\quad (A,\varphi) \mapsto  (A\oplus R[x]/x^2, \varphi \oplus \varphi_0), \\
\sigma^+ \colon  \FGor^{\o,+} \to  \FGor^{\o,+}, &\quad (A,\varphi,e) \mapsto  (A\oplus R[x]/x^2, \varphi \oplus \varphi_0, e\circ \pr_1),
\end{align*}
where $B_{\Hyp} = \left(\begin{smallmatrix} 0 & 1\\ 1& 0\end{smallmatrix}\right)$ is the hyperbolic form and $\phi_0$ is the orientation from Example~\ref{ex:kx/x2}. 
We denote by $\Vect^{\sym,\st}$ the colimit of the sequence
\[
\Vect^\sym\xrightarrow{\tau} \Vect^\sym \xrightarrow{\tau }\Vect^\sym\to\dotsb,
\]
and we define $\FGor^{\o,\st}$ and $\FGor^{\o,+,\st}$ similarly.

By Example~\ref{ex:kx/x2}, the square
\[
\begin{tikzcd}
	\FGor^{\o} \ar{r}{\eta} \ar{d}[swap]{\sigma} & \Vect^{\sym} \ar{d}{\tau} \\
	\FGor^{\o} \ar{r}{\eta} & \Vect^{\sym}
\end{tikzcd}
\]
commutes, inducing a map $\eta^\st\colon \FGor^{\o,\st} \to \Vect^{\sym,\st}$ in the colimit. Similarly, the map  $\theta\colon \FGor^{\ori,+}\to \FGor^{\ori}$ that forgets the augmentation stabilizes to a map $\theta^\st\colon \FGor^{\ori,+,\st}\to \FGor^{\ori,\st}$.

Note that there are canonical maps $\Vect^{\sym,\st}\to\Vect^{\sym,\gp}$ and $\FGor^{\o,\st}\to\FGor^{\o,\gp}$ from the telescopes to the group completions, induced by mapping the $n$th copy of $\Vect^\sym(R)$ resp.\ of $\FGor^{\o}(R)$ to the group completion via $(V,b)\mapsto (V,b)-n\cdot (R^2,B_{\Hyp})$ resp.\ $(A,\phi)\mapsto (A,\phi)-n\cdot (R[x]/x^2,\phi_0)$.
We shall deduce Theorem~\ref{thm:main-gp} from the following variant, which does not involve group completion:

\begin{thm}\label{thm:main-st}
	The map $\eta^\st\colon \FGor^{\o,\st}\to \Vect^{\sym,\st}$ is an $\A^1$-equivalence.
\end{thm}

The strategy of the proof of Theorem~\ref{thm:main-st} is to use the stack $\FGor^{\o,+}$ and show that both $\theta^\st$ and $\eta^\st\circ \theta^\st$ are $\A^1$-equivalences. 
For $\theta^\st$, the idea is to construct an inverse by stabilizing the map 
\[
\gamma\colon\FGor^{\ori}\to \FGor^{\ori,+},\quad (A,\varphi)\mapsto (A\oplus R[x]/x^2, \varphi \oplus \varphi_0, e_0\circ \pr_2).
\]
However, this does not quite work: we have $\theta \circ \gamma \simeq \sigma$, but $\gamma \circ \theta \not\simeq \sigma^+$, since the maps $\gamma \circ \theta$ and $\sigma^+$ equip Gorenstein algebras with different augmentations. Construction~\ref{constr:bankrobbery} allows us to get around this obstruction; it is the main technical ingredient in the proof of Theorem~\ref{thm:main-st}.
 
 For the convenience of the reader, the following table summarizes the various maps we will use in the proof of Theorem~\ref{thm:main-st}:

\begin{center}
  \begin{tabular}{l l}
  name & description\\
    \toprule
         $\eta\colon \FGor^\o \to \Vect^\sym$ & forgets the algebra structure\\
         $\theta\colon \FGor^{\o,+} \to \FGor^\o$ & forgets the isotropic augmentation\\
        $\tau\colon\Vect^\sym \to \Vect^\sym$ & adds a copy of the hyperbolic form\\
    $\sigma\colon\FGor^\o \to \FGor^\o$ & adds a copy of the double point\\
$\sigma^+\colon \FGor^{\o,+}\to \FGor^{\o,+}$ & adds a copy of the double point without changing the augmentation\\
$\gamma\colon\FGor^\o \to \FGor^{\o,+}$ & adds a copy of the double point with its augmentation\\
$\varepsilon\colon\FGor^{\o,+} \to \FGor^{\o,+}$ & takes the connected sum with $R[x]/x^4$ \\
$\pi\colon\FGor^\nonu\to \Vect^\sym$ & forgets the algebra structure
\end{tabular}
\end{center}
\vskip\parskip
    
\begin{constr}
	\label{constr:bankrobbery}
    There is a zigzag of $\A^1$-homotopies 
	 \[
	 \sigma^+ \stackrel{H^\const}\leftsquigarrow \varepsilon \stackrel{H^\mv}\rightsquigarrow \gamma \circ \theta
	 \]
	 and an isomorphism $\psi\colon \theta\circ H^\const\simeq\theta\circ H^\mv$ such that $\psi_0=\id_{\theta\circ\varepsilon}$ and $\psi_1$ is the canonical isomorphism $\theta\circ \sigma^+\simeq \sigma\circ\theta \simeq \theta\circ \gamma\circ \theta$.
\end{constr}

\begin{proof}
    Consider the oriented Gorenstein $\Z[t]$-algebra 
	 \[\robber =  \Z[x,t]/((x-t)^2x^2),\quad \varphi_\robber(r_0 + r_1x + r_2x^2 + r_3x^3) = r_3,\] 
    where $r_i \in \Z[t]$.
	 Its fiber over $t=1$ is isomorphic to $(\Z[x]/x^2,\phi_0)\times(\Z[x]/x^2,\phi_0)$, where $\phi_0$ is the orientation of Example~\ref{ex:kx/x2}. Its fiber over $t=0$ is $\robber_0 =\Z[x]/(x^4)$, with orientation $\phi_{\robber_0}$ given by the same formula as $\phi_\robber$.

    One can view $\robber$ as two copies of the dual numbers colliding at $t = 0$.
    Each copy has a natural augmentation as in
    Example~\ref{ex:dualnumbers}, which extends to an augmentation of $\robber$. Explicitly, we have two augmentations $e^\const, e^\mv\colon \robber \to \Z[t]$ given by sending $x$ to $0$ and to $t$. One computes that
	 \begin{align*}
	 	(e^\const)^*(1) &= t^2x -2tx^2 + x^3,\\
	 	(e^\mv)^*(1) &= -tx^2 + x^3,
	 \end{align*}
	 which shows that both $e^\const$ and $e^\mv$ are isotropic. We thus obtain two elements $(\robber,\varphi_\robber, e^\const)$ and $(\robber,\varphi_\robber,  e^\mv)$ in
    $\FGor^{\ori,+}(\Z[t])$, with local socle generators as above.

	 The augmentations $e^\const$ and $e^\mv$ agree at $t=0$ and define an element $(\robber_0, \varphi_{\robber_0}, 
    e_{\robber_0})\in \FGor^{\ori,+}(\Z)$.
    We let $\varepsilon \colon \FGor^{\ori,+} \to \FGor^{\ori,+}$ be the map that
    sends $(Z,\phi,e)\in \FGor^{\ori,+}(S)$ to its connected sum with
    $(\robber_0, \varphi_{\robber_0}, 
    e_{\robber_0})_{S}$ (see Definition~\ref{def:connected sum}).
	 
	 Given $(Z,\phi,e)\in \FGor^{\ori,+}(S)$, we denote by $(\tilde Z,\tilde \phi,\tilde e^\const)$ the connected sum of $\A^1_Z$ with $(\robber,\phi_\robber,e^\const)_S$ in $\FGor^{\ori,+}(\A^1_S)$. Explicitly, if $s\in\sO(Z)$ is the local socle generator, then $\tilde Z$ is the vanishing locus of the function $(-s, t^2x-2tx^2+x^3)$ on the pushout $\A^1_Z\sqcup_{\A^1_S}(\Spec\robber)_S$, which is defined using $e\colon S\to Z$ and $e^\const\colon \A^1\to\Spec\robber$.
	 Note that replacing $x$ by $t$ in $t^2x-2tx^2+x^3$ gives $0$, so the composite
	 \[
	 \A^1_S \xrightarrow{e^\mv} (\Spec\robber)_S \xrightarrow{\mathrm{can}} \A^1_Z\sqcup_{\A^1_S} (\Spec\robber)_S
	 \]
	 lands in $\tilde Z$. This defines another section $\tilde e^\mv\colon \A^1_S\to\tilde Z$, which is moreover isotropic. 
	 Indeed, we have $(\tilde e^\mv)^*(1)=(0,(e^\mv)^*(1))$, hence $\tilde e^\mv((\tilde e^\mv)^*(1))=e^\mv((e^\mv)^*(1))=0$.
	 
    Let $H^{\const}\colon \FGor^{\ori,+}\to \FGor^{\ori,+}(\A^1\times -)$ be the map that
    sends $(Z,\phi,e)$ to $(\tilde Z,\tilde \phi,\tilde e^\const)$, and let $H^{\mv}\colon \FGor^{\ori,+}\to \FGor^{\ori,+}(\A^1\times -)$ be the map that
    sends $(Z,\phi,e)$ to $(\tilde Z,\tilde \phi,\tilde e^\mv)$.
	 Then it is clear that $H^\const_0=H^\mv_0\simeq\varepsilon$.
	 Moreover, we have $H^\const_1\simeq\sigma^+$ and $H^\mv_1\simeq\gamma\circ\theta$.
	 Indeed, the fiber of $(\tilde Z,\tilde\phi)$ over $t=1$ is the disjoint union of $(Z,\phi)$ and $(\Z[x]/x^2,\phi_0)_S$, with $\tilde e^\const_1$ being the given section $e$ to the first summand and $\tilde e^\mv_1$ the canonical section to the second summand.
	 Thus, $H^\const$ and $H^\mv$ are the desired $\A^1$-homotopies. By construction, the underlying oriented Gorenstein schemes of $H^\const(Z,\phi,e)$ and $H^\mv(Z,\phi,e)$ are the same, and we can take the isomorphism $\psi$ to be the identity.
\end{proof}

\begin{prop}\label{prop:bank-robbery}
	The map $\theta^\st\colon \FGor^{\ori,+,\st}\to \FGor^{\ori,\st}$ is an $\A^1$-equivalence.
\end{prop}

\begin{proof}
	Consider the diagram
	\[
	\begin{tikzcd}
	    \FGor^{\ori,+} \ar["\theta"]{r} \ar[swap,"\sigma^+"]{d} & \FGor^{\ori} \ar["\sigma"]{d} \ar[swap,"\gamma"]{dl}  \\
		\FGor^{\ori,+}\ar["\theta"]{r} & \FGor^{\ori} \rlap.
	\end{tikzcd}
	\]
	By Construction~\ref{constr:bankrobbery} there is a homotopy $H\colon \Lhtp(\sigma^+)\simeq \Lhtp(\gamma\circ\theta)$ such that $\theta\circ H$ is the canonical isomorphism $\theta\circ \sigma^+\simeq \sigma\circ\theta\simeq \theta\circ\gamma\circ\theta$. It follows that the map $\Lhtp\gamma$ induces in the colimit a map $\Lhtp\FGor^{\o,\st} \to \Lhtp\FGor^{\o,+,\st}$, which is inverse to $\Lhtp\theta^\st$.
\end{proof}

\begin{prop}\label{prop:vect-nu}
	The forgetful map $\pi\colon \FGor^{\nonu}\to \Vect^\sym$ is an $\A^1$-equivalence.
\end{prop}

\begin{proof}
	Let $\nu\colon \Vect^\sym \to\FGor^{\nonu}$ be the map sending a symmetric space $(V,B)$ to the non-unital oriented Gorenstein algebra $(V, B)$ with zero multiplication. 
	 Then $\pi\circ\nu$ is the identity, and the map
	\[
	\FGor^{\nonu} \to \FGor^{\nonu}(\A^1\times-), \quad (V, B) \mapsto (tV[t], (tp,tq)\mapsto B_{R[t]}(p,q)),
	\]
	where $B_{R[t]}$ is the $R[t]$-bilinear extension of $B$, is an $\A^1$-homotopy from $\nu\circ\pi$ to the identity of $\FGor^{\nonu}$.
\end{proof}

Recall that a symmetric space $(V,b)$ is said to be \emph{metabolic} if it has a Lagrangian, i.e., a direct summand $L\subset V$ such that $L=L^\perp$. For example, for every $V\in\Vect(R)$, the hyperbolic space $\Hyp V=\left(V\oplus V^\vee,\left(\begin{smallmatrix} 0 & I\\ I &0\end{smallmatrix}\right)\right)$ is metabolic, with Lagrangian $V\oplus 0 \subset V\oplus V^\vee$. When 2 is invertible, all metabolic spaces are in fact of this form. The following lemma shows that this is also the case up to $\A^1$-homotopy, even when 2 is not invertible.
    
\begin{lem}\label{lem:metabolic=hyperbolic}
	Let $(V,b)$ be a symmetric space over a ring $R$.
   If $(V,b)$ is metabolic with Lagrangian $L$, then the class of $(V,b)$ in $\pi_0(L_{\A^1}\Vect^\sym)(R)$ is equal to the class of $\Hyp L$.
\end{lem}

\begin{proof}
Let $W$ be a complement of $L$ in $V$. Then the map $W\to V\simeq V^\vee\to L^\vee$ is an isomorphism and we can identify $V$ as $L\oplus L^\vee$. Under this decomposition $b$ is given by the matrix
\[\begin{pmatrix} 0 & I\\ I & A\end{pmatrix},\]
where $A$ is some symmetric bilinear form on $L^\vee$. Then
\[\left(L[t]\oplus L[t]^\vee,\ \begin{pmatrix} 0 & I\\ I & tA\end{pmatrix}\right)\in \Vect^\sym(R[t])\]
is an $\A^1$-homotopy between $(V,b)$ and $\Hyp L$.
\end{proof}

\begin{lem}\label{lem:invert-hyp}
    Let $\Hyp=\Hyp \Z\in \Vect^\sym(\Z)$, and let $\Vect^\sym[-\Hyp]$ be the commutative monoid obtained from $\Vect^\sym$ by additively inverting $\Hyp$.
	 \begin{enumerate}
		\item The cyclic permutation of $\Hyp^{\oplus 3}$ is $\A^1$-homotopic to the identity.
	 	\item The canonical map $\Vect^\sym[-\Hyp]\to \Vect^{\sym,\gp}$ is an $\A^1$-equivalence on affine schemes.
	 	\item The canonical map $\Vect^{\sym,\st}\to\Vect^{\sym,\gp}$ is an $\A^1$-equivalence on affine schemes.
	 \end{enumerate}
\end{lem}

\begin{proof}
	(1) The cyclic permutation of $\Hyp^{\oplus 3}$ is given by applying the functor $\Hyp$ to the cyclic permutation of $\Z^3$ in $\Vect$, which is $\A^1$-homotopic to the identity (being a product of elementary matrices).
	
    (2) It suffices to show that every element of $\pi_0(L_{\A^1}\Vect^\sym[-\Hyp])(R)$ is additively invertible. Every object $(V,b)\in\Vect^\sym(R)$ is a summand of $(V,b)\oplus (V,-b)$, which is metabolic with Lagrangian given by the diagonal copy of $V$, so it suffices to show that every metabolic object is invertible. By Lemma~\ref{lem:metabolic=hyperbolic}, it then suffices to show that every object of the form $\Hyp V$ is invertible. But if $W$ is such that $V\oplus W=R^n$, we have $\Hyp V\oplus \Hyp W=\Hyp(R^n)=(\Hyp R)^{\oplus n}$, and so $\Hyp V$ is invertible.
	 
	 (3) This follows from (1) and (2) by \cite[Proposition 5.1]{deloop4}.
\end{proof}

\begin{proof}[Proof of Theorem~\ref{thm:main-st}]
	By Lemma~\ref{lem:local_socle}, there is a commutative square
	\[
	\begin{tikzcd}
	\FGor^{\o,+} \ar{r}{\theta} \ar{d}[swap]{(\id\times\pi)\circ\Aug} & \FGor^\o \ar{d}{\eta} \\
	\A^1\times \Vect^\sym \ar{r}{\tilde\tau} & \Vect^\sym\rlap,
	\end{tikzcd}
	\]
	where $\Aug\colon \FGor^{\o,+}\to\A^1\times \FGor^{\nonu}$ was defined in Proposition~\ref{prop:nonunital=augmented} and $\tilde\tau(\lambda,-)$ adds a copy of the bilinear form $\left(\begin{smallmatrix}\lambda & 1 \\ 1 & 0 \end{smallmatrix}\right)$.
	Moreover, this square fits in a commutative cube with the respective ``stabilization maps'' $\sigma^+$, $\sigma$, $\id_{\A^1}\times\tau$, and $\tau$, and in the colimit we obtain a commutative square
	\[
	\begin{tikzcd}
	\FGor^{\o,+,\st} \ar{r}{\theta^\st} \ar{d}[swap]{((\id\times\pi)\circ\Aug)^\st} & \FGor^{\o,\st} \ar{d}{\eta^\st} \\
	\A^1\times \Vect^{\sym,\st} \ar{r}{\tilde\tau^\st} & \Vect^{\sym,\st}\rlap.
	\end{tikzcd}
	\]
	By Propositions \ref{prop:nonunital=augmented} and~\ref{prop:vect-nu}, the left vertical map is an $\A^1$-equivalence (already in the unstable square).
	By Proposition~\ref{prop:bank-robbery}, $\theta^\st$ is an $\A^1$-equivalence. We have $\tilde\tau^\st(0,-)=\tau^\st$, where $\tau^\st\colon \Vect^{\sym,\st}\to\Vect^{\sym,\st}$ is the action of $\Hyp\in\Vect^\sym(\Z)$ on the $\Vect^\sym$-module $\Vect^{\sym,\st}$.  Note that $\tau^\st$ is \emph{not} an equivalence, but it is an $\A^1$-equivalence by \cite[Proposition 5.1]{deloop4} since the cyclic permutation of $\Hyp^{\oplus 3}$ is $\A^1$-homotopic to the identity (Lemma~\ref{lem:invert-hyp}(1)). Thus, $\tilde\tau^\st$ is also an $\A^1$-equivalence. We conclude that $\eta^\st$ is an $\A^1$-equivalence, as desired.
\end{proof}

\begin{proof}[Proof of Theorem~\ref{thm:main-gp}]
	As above, let $\Vect^\sym[-\Hyp]$ be obtained from $\Vect^\sym$ by additively inverting $\Hyp$, and let $\FGor^\o[-\Z[x]/x^2]$ be obtained from $\FGor^\o$ by additively inverting the oriented Gorenstein algebra $(\Z[x]/x^2, \varphi_0)$ from Example~\ref{ex:kx/x2}. We have a commutative square 
		\[
	\begin{tikzcd}
		 \FGor^{\o,\st} \ar{r}   \ar[swap,"\eta^\st"]{d}  & \FGor^\o[-\Z[x]/x^2] \ar{d} \\
		\Vect^{\sym,\st} \ar{r} & \Vect^\sym[-\Hyp] \rlap.
	\end{tikzcd}
	\]
	By Lemma~\ref{lem:invert-hyp}(1), the cyclic permutation of $\Hyp^{\oplus 3}$ becomes the identity in $\Lhtp\Vect^\sym$, which by \cite[Proposition 5.1]{deloop4} implies that the lower horizontal map is an $\A^1$-equivalence.
	By Theorem~\ref{thm:main-st}, the left vertical map is an $\A^1$-equivalence. In particular, the cylic permutation of $(\Z[x]/x^2)^{\times 3}$ also becomes the identity in $\Lhtp\FGor^{\o,\st}$. It then follows again from \cite[Proposition 5.1]{deloop4} that the upper horizontal map is an $\A^1$-equivalence. Hence, the right vertical map is an $\A^1$-equivalence.
	 Since the functor $\Lhtp$ commutes with group completion \cite[Lemma 5.5]{HoyoisCdh}, we deduce that $ \FGor^{\o,\gp}\to\Vect^{\sym,\gp}$ is an $\A^1$-equivalence.
\end{proof}

\begin{rem}\label{rem:FGor-gp}
	Combining Theorems~\ref{thm:main-gp} and \ref{thm:main-st} with Lemma~\ref{lem:invert-hyp}(3), we deduce that the canonical map $\FGor^{\o,\st}\to \FGor^{\o,\gp}$ is an $\A^1$-equivalence on affine schemes.
\end{rem}

\section{Consequences in complexity theory}

In this short section we derive consequences of Proposition~\ref{prop:vect-nu}
for structure tensors of finite algebras.  This section is an interesting side
application of our methods: it has no relation to the subsequent sections, however it is of
interest for complexity theory. Because of this target audience,
we strive to be more explicit than elsewhere. In this section we work over an
algebraically closed field $k$.

For $q\geq 1$ the \emph{G-fat} point~\cite{Casnati_Notari_6points} is the spectrum of a
locally free $k$-algebra $A_q$ of degree $q+2$ presented as
\[
    A_q := \frac{k[y_1, \ldots ,y_q]}{(y_iy_j\ |\ i\neq j) + (y_i^2 - y_j^2\
    |\ i\neq j)+(y_1^3)}.
\]
This is an oriented Gorenstein algebra with orientation $\varphi_0 :=
(y_1^2)^*\in \Hom_k(A_q, k)$, and its unique augmentation is isotropic by
Example~\ref{ex:isotropic}. In the equivalence from
Proposition~\ref{prop:nonunital=augmented} the triple $(A_q, \varphi, e)$
corresponds to $(0, V_0)$, where $V_0$ is spanned by self-dual elements $y_1,
\ldots ,y_q$ and equipped with a trivial multiplication.

Let $(C, \varphi, e)$ be another isotropically augmented Gorenstein algebra of
degree $q+2$.
As in Definition~\ref{dfn:oriented-Gorenstein-algebra}, we have a
non-degenerate form $B_{\varphi}$ on $C$ and a local socle generator $x\in C$.
(If $C$ is local, then $x$ is exactly its socle generator, see
Lemma~\ref{lem:local_socle}). Let $V = 1_C^{\perp}\cap x^{\perp} \subset C$.
Then $B_{\varphi}$
restricts to a non-degenerate form on $V$, so as a $k$-vector space we have $C = k\cdot 1_C\oplus kx \oplus V$.

A $\G_m$-equivariant \emph{degeneration} of a finite $k$-algebra $C$ to a finite $k$-algebra $A$
is a finite flat family $f\colon \mathcal{X}\to \A^1$ together with a
$\G_m$-action on $\mathcal{X}$ making $f$ equivariant with respect to the usual
$\G_m$-action on $\A^1$ and such that $\mathcal{X}_1  \simeq \Spec(C)$ and
$\mathcal{X}_0  \simeq \Spec(A)$. The $\G_m$-equivariance implies that the fiber
of $\mathcal{X}$ over every nonzero $k$-point of $\A^1$ is isomorphic to $\Spec(C)$.

Proposition~\ref{prop:vect-nu} and Proposition~\ref{prop:nonunital=augmented}
contain a construction of a degeneration of $C$ to $A_q$. Explicitly, it is given by $f\colon\Spec(\mathcal{C})\to
\Spec(k[t])$, where $\mathcal{C}$
is a free $k[t]$-module $(k[t]\cdot 1 \oplus k[t]x)\oplus V[t]$ with
multiplication given by
\[
    (r+sx,v)\cdot (r'+s'x,v')=(rr'+(sr'+s'r+B_{\varphi}(v,v'))x, r'v+rv'+tvv'),
\]
where $v,v'\in V[t]$ and $B_{\varphi}$ extends $k[t]$-linearly.
The algebra $\mathcal{C}$ admits a $\G_m$-action given by $t\cdot (r+sx, v) =
(r+st^{-2}x, t^{-1}v)$ which proves that $f$ is a degeneration.
\begin{prop}\label{prop:abstractDegenerations}
    Let $q\geq 1$ and let $C$ be a Gorenstein $k$-algebra of degree $q+2$. Then $C$ admits a $\G_m$-equivariant 
    degeneration to $A_q$.
\end{prop}
\begin{proof}
    Choose an orientation of $C$. If $C$ has nonzero nilpotent radical, then by
    Example~\ref{ex:isotropic} it admits an isotropic augmentation and the
    associated degeneration $\mathcal{C}$ above proves our claim.
    It remains to consider reduced $C$. Since $k$ is algebraically closed, we
    have $C = \prod_{q+2} k$.
    In this case, choose a set $\Gamma$ of $q+2$ general lines through the origin of
    $\A^{q+1}$ and a hyperplane $H = (y = 1)$ intersecting them transversely.
    Then $\Gamma\cap (y = t)$ is a degeneration over $\A^1$ with
    parameter $t$. Since a general tuple of $q+2$ points on $\P^{q}$ is
    arithmetically Gorenstein, the special fiber of the degeneration is
    Gorenstein. This fiber has $q$-dimensional tangent space, which implies
    that it is $\Spec(A_q)$.
\end{proof}
\begin{rem}
    Proposition~\ref{prop:abstractDegenerations} in the special case when $C$
    is local was obtained
    in~\cite{Casnati_Notari_6points}, by completely different means.
    This proposition also shows that the Gorenstein locus of the Hilbert scheme of
    $d$ points on $\A^n$ is
    connected by rational curves whenever $n\geq d-2$. In general, even
    topological connectedness of the Gorenstein locus is open.
\end{rem}
For the terminology on tensors below, we refer
to~\cite[5.6.1]{Landsberg__complexity_book}.
\begin{cor}\label{cor:degTensors}
    Let $m\geq 3$ and let $T\in k^m\otimes k^m\otimes k^m$ be a $1$-generic tensor of minimal
    border rank (that is, of border rank $m$). Then $T$ degenerates to the big
    Coppersmith--Winograd tensor $\mrm{CW}_{m-2}$.
\end{cor}
\begin{proof}
    As explained in~\cite[5.6.2.1]{Landsberg__complexity_book}, the tensor $T$
    is isomorphic to a structure tensor of a Gorenstein algebra $C$. The tensor
    $\mrm{CW}_{m-2}$ is isomorphic to the structure tensor of $A_{m-2}$. The claim
    follows directly from Proposition~\ref{prop:abstractDegenerations}.
\end{proof}
\begin{rem}
    The assumptions of Corollary~\ref{cor:degTensors} can be much weakened. We
    only need to assume that $T$ is $1$-generic and satisfies Strassen's
    equations, see~\cite[\sectsign 2.1]{Landsberg_Michalek__Abelian_Tensors} for
    their definition. Indeed, to such a $T$ we associate a commuting tuple of
    matrices~\cite[Definition 2.7]{Landsberg_Michalek__Abelian_Tensors}, hence a
    module over a polynomial ring~\cite[Introduction]{Jelisiejew_Sivic}. Since $T$ is $1$-generic, it
    is $1_B$-generic, hence this module is cyclic and can be viewed as a $k$-algebra.
    By~\cite[5.6.2.1]{Landsberg__complexity_book} this algebra is
    Gorenstein and we argue as in the proof of~Corollary~\ref{cor:degTensors}.
\end{rem}

\section{Oriented Hilbert scheme and orthogonal Grassmannian}

\begin{defn}
    Let $X\to S$ be a morphism of schemes. The \emph{oriented Hilbert scheme} 
	 \[\Hilb^{\Gor,\o}(X/S)\colon \Sch_S^\op\to\Set\]
	 is the pullback
    \[\Hilb^{\Gor,\o}(X/S):=\Hilb(X/S)\times_{\FFlat} \FGor^\o\,.\]
    That is, $\Hilb^{\Gor,\o}(X/S)$ classifies finite locally free subschemes of $X$ that are Gorenstein and equipped with an orientation.
\end{defn}

\begin{rem}
    If $X\to S$ is such that $\Hilb(X/S)$ is a scheme (for example, $X$ is quasi-projective over $S$) or an algebraic space (for example, $X$ is separated over $S$), then so is $\Hilb^{\Gor,\o}(X/S)$. Indeed, the forgetful map $\Hilb^{\Gor,\o}(X/S)\to \Hilb(X/S)$ is a base change of $\FGor^\o\to \FFlat$, which is representable by schemes by Remark~\ref{rem:FGoror}.
\end{rem}

From now on we will be interested in the ind-scheme $\Hilb^{\Gor,\o}(\A^\infty) := \colim_n \Hilb^{\Gor,\o}(\A^n)$ over $\Spec\Z$. 
Note that we have a coproduct decomposition
\[\Hilb^{\Gor,\o}(\A^\infty)\simeq\coprod_{d\ge0} \Hilb^{\Gor,\o}_d(\A^\infty)\]
where $\Hilb^{\Gor,\o}_d(\A^\infty)$ classifies the oriented Gorenstein subschemes of $\A^\infty$ of degree $d$.

\begin{prop}\label{prop:Hilb-FGor}
The forgetful map $\Hilb^{\Gor,\o}(\A^\infty)\to \FGor^{\o}$ is a universal $\A^1$-equivalence on affine schemes.
\end{prop}

\begin{proof}
	By definition, this map is a base change of the forgetful map $\Hilb(\A^\infty)\to \FFlat$, which is a universal $\A^1$-equivalence on affine schemes by~\cite[Proposition~4.2]{robbery}.
\end{proof}

Consider the map $\sigma\colon\Hilb^{\Gor,\o}(\A^\infty)\to \Hilb^{\Gor,\o}(\A^\infty)$ sending an oriented Gorenstein subscheme $Z\subset \A^\infty$ to 
\[(\Spec \Oo[x]/x^2\times\{0\})\sqcup (\{0\}\times Z)\subset \A^1\times \A^\infty,\]
where $\Spec \Oo[x]/x^2$ is equipped with the orientation of Example~\ref{ex:kx/x2} and is embedded in $\A^1$ at $1$. Then we define
\[\Hilb^{\Gor,\o}_\infty(\A^\infty) := \colim(\Hilb^{\Gor,\o}_2(\A^\infty)\xrightarrow{\sigma}\Hilb^{\Gor,\o}_4(\A^\infty)\xrightarrow{\sigma}\cdots).\]

\begin{cor}\label{cor:Hilb-FGor}
    The forgetful maps
	 \[
	 \uZ\times \Hilb^{\Gor,\o}_\infty(\A^\infty) \to \FGor^{\o,\gp} \to \Vect^{\sym,\gp}
	 \]
    are $\A^1$-equivalences on affine schemes.
\end{cor}

\begin{proof}
	The second map is an $\A^1$-equivalence on all schemes by Theorem~\ref{thm:main-gp}.
	The first map factors as
	\[
	\uZ\times \Hilb^{\Gor,\o}_\infty(\A^\infty) \to \FGor^{\o,\st} \to \FGor^{\o,\gp},
	\]
	where the second map is an $\A^1$-equivalence on affine schemes by Remark~\ref{rem:FGor-gp}.
	On quasi-compact schemes, the presheaf $\uZ\times \Hilb^{\Gor,\o}_\infty(\A^\infty)$ is the colimit of
	\[\Hilb^{\Gor,\o}(\A^\infty)\xrightarrow{\sigma}\Hilb^{\Gor,\o}(\A^\infty)\xrightarrow{\sigma}\cdots,\]
	while $\FGor^{\o,\st}$ is the colimit of
	\[\FGor^\o\xrightarrow{\sigma}\FGor^\o\xrightarrow{\sigma}\cdots.\]
	Hence, the result follows from Proposition~\ref{prop:Hilb-FGor}.
\end{proof}

Let $(V,B)\in \Vect^\sym(S)$ be a vector bundle over $S$ with a non-degenerate symmetric bilinear form. Recall from \cite{SchlichtingTripathi} that  the \emph{orthogonal Grassmannian} $\GrO_d(V,B)$ is the open subscheme of the Grassmannian $\Gr_d(V)$ given by the rank $d$ subbundles $W\subset V$ such that the restriction of $B$ to $W$ is non-degenerate. We consider the smooth ind-scheme
\[\GrO_d=\GrO_d(\Hyp^\infty):=\colim_n \GrO_d(\Hyp^n)\]
over $\Spec\Z$. Forgetting the embedding into $\Hyp^\infty$ defines a canonical map $\GrO_d\to \Vect^\sym_d$.

A bilinear form $B$ on an $R$-module $V$ is called \emph{even} if it is symmetric and $B(x,x)\in 2R$ for all $x\in V$. Clearly, if $2\in R^\times$, every symmetric form is even. If $V$ is projective, a symmetric bilinear form $B$ on $V$ is even if and only if there exists a bilinear form $B'$ on $V$ such that $B(x,y)=B'(x,y)+B'(y,x)$.

\begin{lem}\label{lem:Hyp-embedding}
	Let $R$ be a ring, $V$ a finite locally free $R$-module, and $B$ an even symmetric bilinear form on $V$.
	Then there exists an isometric embedding $V\hookrightarrow \Hyp(R^n)$ for some $n\geq 0$.
\end{lem}

\begin{proof}
	Choose an isomorphism $V\oplus W\simeq R^n$ and let $A$ be the matrix representing $B'\oplus 0$. Then we can take the composition
	\[
	V\hookrightarrow R^n \to \Hyp(R^n),
	\]
	where the second map has components $(A, \id_{R^n})$.
\end{proof}

Note that every hyperbolic space $\Hyp V$ is even. 
It follows that the forgetful map $\GrO_d\to\Vect^\sym_d$ lands in the subpresheaf $\Vect^\ev_d\subset \Vect^\sym_d$ of even forms.

\begin{prop}\label{prop:Vect^bil and Gr^o}
    For any $d\ge 0$, the forgetful map
    \[\GrO_d\to \Vect^\ev_d\]
    is a universal $\A^1$-equivalence on affine schemes.
\end{prop}

\begin{proof}
    We will apply \cite[Lemma~4.1(2)]{robbery}, or rather its proof. 
	 Unlike $\Vect^\sym_d$, the presheaf $\Vect^\ev_d$ for $d\geq 2$ does not satisfy closed gluing (in particular, it is not an algebraic stack). 
	 Indeed, consider a closed pushout of affine schemes $X\sqcup_ZY$ on which there is a noneven function $f$ whose restriction to both $X$ and $Y$ is even (for example: $X=Y=\Spec\Z[x,y]$, $Z=\Spec\Z[x,y]/(2x-2y)$, and $f=(2x,2y)$); then the non-degenerate symmetric bilinear form $\left(\begin{smallmatrix} f & 1\\ 1 & 0\end{smallmatrix}\right)$ is not even, even though its restriction to both $X$ and $Y$ is. 
	 However, $\Vect^\ev_d$ does nevertheless transform the colimit $\partial \A^n_R=\colim_{\Delta^k\hook \partial\Delta^n} \A^k_R$ into a limit, since the presheaf $\Spec R\mapsto 2R$ does.
	 It therefore suffices to check that for every commutative square
    \[\begin{tikzcd}
        \Spec R/I \ar[r] \ar[d] & \GrO_d\ar[d]\\
        \Spec R\ar[r]\ar[ur,dashed] & \Vect^\ev_d
    \end{tikzcd}\]
    with $R=R_0[t_0,\dotsc,t_n]/(\sum_i t_i-1)$ and $I=(t_0\dotsm t_n)$, there exists a diagonal arrow making both triangles commute. Unwrapping the definitions of $\Vect^\ev_d$ and $\GrO_d$, this means that for every even bilinear space $(V,B)$ over $R$ and every isometric embedding $F\colon V/IV\hook \Hyp(R/I)^N$, we must find an isometric embedding $\tilde F\colon V\hook \Hyp(R)^N$ lifting $F$, after possibly increasing $N$.
    
    First, let $\tilde F\colon V\to \Hyp(R)^N$ be any lift of $F$ (possibly not isometric), which exists since $V$ is projective over $R$. If we write
    \[\tilde B(x,y)\coloneqq B(x,y)-B_{\Hyp}\left(\tilde Fx, \tilde Fy\right),\]
    then $\tilde B$ is an $I$-valued even form on $V$. Our first claim is that we can change $\tilde F$ so that $\tilde B$ has values in $I^2$. 
	 To do so, note that since $I\cap 2R=2I$, every $I$-valued even form on $V$ is an $I$-linear combination of even forms. Hence, we can find an $I$-valued bilinear form $\tilde B'$ such that $\tilde B(x,y)=\tilde B'(x,y)+\tilde B'(y,x)$.
	 Since $B$ is non-degenerate, there exists $r\colon V\to IV$ such that
     \[\tilde B'(x,y)=B(x,r(y))\,.\]
	Then it is easy to check that replacing $\tilde F$ with $\tilde F + \tilde Fr$ has the desired effect.
    
    Without loss of generality, we can therefore assume that $\tilde B$ takes values in $I^2$. Since $I^2\cap 2R=2I^2$, we can then write
    \[\tilde B=\sum_i a_i b_i C_i\]
    for some $a_i, b_i\in I$ and some even bilinear forms $C_i\colon V\times V\to R$. 
	 Using Lemma~\ref{lem:Hyp-embedding}, we can find for each $i$ an $R$-linear map $G_i\colon V\to \Hyp(R)^{N_i}$ such that 
    \[C_i (x,y)=B_{\Hyp}(G_ix, G_iy)\,.\]
	 Then an isometric lift of $F$ is given by
    \[\left(\tilde F,\left((a_i,b_i)\circ G_i\right)_i\right)\colon V\to \Hyp(R)^{N+\sum_i N_i}\,,\]
    where $(a_i,b_i)\colon\Hyp(R)^{N_i}\to \Hyp(R)^{N_i}$ is the map given on each component by the matrix $\left(\begin{smallmatrix} a_i & 0\\ 0 & b_i\end{smallmatrix}\right)$.
\end{proof}

Let $\tau\colon\GrO_d\to \GrO_{d+2}$ be the map sending $V\subset \Hyp^\infty$ to $\Hyp \oplus V\subset \Hyp^{1+\infty}$, and define the infinite orthogonal Grassmannian as
\[\GrO_\infty\coloneqq \colim(\GrO_2\xrightarrow\tau \GrO_4\xrightarrow\tau \cdots)\,.\]

Using Proposition~\ref{prop:Vect^bil and Gr^o} we obtain the following result, which generalizes a theorem of Schlichting and Tripathi \cite[Theorem~5.2]{SchlichtingTripathi} to rings that are not necessarily regular nor $\Z[\tfrac 12]$-algebras.

\begin{cor}\label{cor: Vect^bil and GrO}
   The forgetful map
 \[
 \uZ\times \GrO_\infty \to \Vect^{\ev,\gp}
 \]
   is an $\A^1$-equivalence on affine schemes.
\end{cor}

\begin{proof}
	This map factors as
	\[
	\uZ\times \GrO_\infty \to \Vect^{\ev,\st} \to \Vect^{\ev,\gp},
	\]
	where the second map is an $\A^1$-equivalence on affine schemes by Lemma~\ref{lem:invert-hyp} (in which we may replace $\Vect^\sym$ by $\Vect^\ev$).
	On quasi-compact schemes, the presheaf $\uZ\times \GrO_\infty$ is the colimit of
	\[\coprod_{d\ge0}\GrO_d\xrightarrow{\tau}\coprod_{d\ge0}\GrO_d\xrightarrow{\tau}\cdots,\]
	where the coproducts are taken in $\Pre_\Sigma(\Sch)$.
	On the other hand, $\Vect^{\ev,\st}$ is the colimit of
	\[\Vect^\ev\xrightarrow{\tau}\Vect^\ev\xrightarrow{\tau}\cdots.\]
	Hence, the result follows from Proposition~\ref{prop:Vect^bil and Gr^o}.
\end{proof}

\begin{rem}
	The spaces $\Vect^\sym(R)^\gp$ and $\Vect^\ev(R)^\gp$ appearing in Corollaries \ref{cor:Hilb-FGor} and~\ref{cor: Vect^bil and GrO} are equivalent to the ``genuine symmetric'' and ``genuine even'' Grothendieck--Witt spaces $\GWspace^\mathrm{gs}(R)$ and $\GWspace^\mathrm{ge}(R)$ defined in \cite[Section 4]{HermitianII}; this is proved in \cite{HebestreitSteimle}. 
	We note that the forgetful map $\GWspace^\mathrm{ge}\to\GWspace^\mathrm{gs}$ does not induce an equivalence in $\H(S)$ if $2$ is not invertible on $S$. Both motivic spaces are stable under base change (by Corollary~\ref{cor: Vect^bil and GrO} and \cite[Example A.0.6(5)]{deloop3}, respectively), so it suffices to consider a field $k$ of characteristic $2$. Then every even form over $k$ is alternating (by definition), and non-degenerate alternating forms have even rank \cite[I, Corollary 3.5]{HM}, so the composition
	\[
	\pi_0\GWspace^\mathrm{ge}(k)\twoheadrightarrow\pi_0(L_\mot\GWspace^\mathrm{ge})(k)\to \pi_0(L_\mot\GWspace^\mathrm{gs})(k)\stackrel{\rk}\twoheadlongrightarrow \Z
	\]
	has image $2\Z\subset\Z$.
\end{rem}

\section{Isotropic Gorenstein algebras}
\label{sec:isotropic}

Inspired by~\cite[Theorem~2.1]{robbery}, one is tempted to ask whether the map
\[\alpha\colon\Vect^\sym_{d-2}\to \FGor^\o_d\]
sending $(V,B)$ to the oriented Gorenstein algebra $(R[x]/x^2\oplus V,\phi)$ of Example~\ref{ex:Gorenstein-algebras-from-bilinear-forms} is an $\A^1$-equivalence. Unfortunately this turns out not to be the case.

\begin{prop}\label{prop: square zero ext over R}
 For any formally real field $k$ and any $d\geq 2$, the map
 \[\alpha\colon\Vect^\sym_{d-2}\to \FGor^\o_d\]
 is not a motivic equivalence on smooth $k$-schemes.   
\end{prop}

\begin{proof}
 Let us consider the composition
 \[\FGor^\o_d\to \Vect^\sym_d\to \uW\,,\]
 where $\uW$ is the Zariski sheafification of the presheaf of Witt groups on smooth $k$-schemes.
 It is well-known that $\uW$ is an $\A^1$-invariant Nisnevich sheaf, for example by \cite[Theorem 3.4.11]{deloop1}, since Witt groups have framed transfers and are $\A^1$-invariant on regular noetherian $\Z[\tfrac 12]$-schemes (see the discussion in the introduction of \cite{GilleWittGroups}). Hence, this composition factors through a map
 \[\pi_0\left(L_\mot\FGor^\o_d\right)\to \uW\,.\]
 Let us now consider the composition
 \[\pi_0\left(L_\mot\Vect^\sym_{d-2}\right)(k)\xrightarrow\alpha \pi_0\left(L_\mot\FGor^\o_d\right)(k)\to \uW(k)\to \Z\,,\]
 where the last map is the signature associated with some real closure of $k$. We claim that the image of this composition is contained in the subset $\{n\in\Z\mid -d+2\le n\le d-2\}$. Indeed, $\pi_0\left(L_\mot\Vect^\sym_{d-2}\right)(k)$ is a quotient of $\pi_0\Vect^\sym_{d-2}(k)$ \cite[\sectsign 2, Corollary 3.22]{MV}, so it suffices to consider the image of the latter. By definition of $\alpha$, the composition 
 \[
 \Vect^\sym_{d-2}\xrightarrow{\alpha}\FGor^\o_d\to \Vect^\sym_d
 \]
adds a hyperbolic form, and in particular does not change the signature, which is therefore bounded by the rank $d-2$.
 
 Now if $\alpha$ were a motivic equivalence, this would imply that the image of
 \[\pi_0\left(L_\mot\FGor^\o_d\right)(k)\to \uW(k)\to \Z\]
 is also contained in the subset of integers of absolute value $\leq d-2$.
 But the oriented Gorenstein algebra $k^d$ with the orientation  $\varphi(x_1,\dots,x_d)=x_1+\cdots+x_d$ has signature $d$, thus providing a contradiction.
\end{proof}

\begin{rem}
    The algebra $k^d$ seems to be the only example leading to a contradiction in Proposition~\ref{prop: square zero ext over R}.
    More precisely, consider the complement $\mathcal{Z}$ of the open substack
    of smooth schemes in
    $\FGor^\o_d$. The map $\alpha$ factors through $\mathcal{Z}$, and it seems
    possible that $\alpha\colon\Vect^\sym_{d-2}\to \mathcal{Z}$ is an
    $\A^1$-equivalence.
\end{rem}

We now modify Proposition~\ref{prop: square zero ext over R} so that it becomes a positive statement.

\begin{defn}
An oriented Gorenstein $R$-algebra $(A,\varphi)$ is called \emph{isotropic} if $\varphi(1)=0$. We write $\FGor^{\o,0}_d\subset \FGor^\o_d$ for the substack of isotropic oriented Gorenstein algebras of rank $d$.
\end{defn}

\begin{rem}
	We warn the reader that the notion of isotropic oriented Gorenstein algebra is not directly related to that of isotropically augmented oriented Gorenstein algebra from Definition~\ref{def:isotropic-aug}. In particular, the forgetful functor $\FGor^{\o,+}\to\FGor^\o$ does not land in $\FGor^{\o,0}$. Nevertheless, there is a zigzag of $\A^1$-equivalences
	\[
	\FGor^{\o,0}_{\geq 2} \leftarrow \FGor^{\o,0}\times_{\FGor^\o} \FGor^{\o,+} \to \FGor^{\o,+},
	\]
	where the middle term is equivalent to $\FGor^\nonu$. Indeed, the right-hand map can be indetified with the zero section $\FGor^\nonu\hook \A^1\times\FGor^\nonu$ via Proposition~\ref{prop:nonunital=augmented}, and the left-hand map fits in a commuting triangle
	\[
	\begin{tikzcd}
		\FGor^{\o,0}_{\geq 2} \ar{dr} &[-35pt] &[-35pt] \FGor^\nonu \ar{ll} \ar{dl}{\pi} \\
		& \Vect^\sym &
	\end{tikzcd}
	\]
	where the diagonal maps are $\A^1$-equivalences (Propositions \ref{prop:vect-nu} and \ref{prop: Vect^bil and FGor}).
\end{rem}

Note that the image of the map $\alpha$ lies in the substack $\FGor^{\o,0}_d\subset \FGor^{\o}_d$. 
Moreover, for $(A,\varphi) \in \FGor^{\o,0}_d(R)$, the underlying symmetric bilinear form $(A,B_\varphi)$ is equipped with a canonical isotropic subspace $R\subset A$. Therefore we can apply algebraic surgery to it and obtain the non-degenerate symmetric bilinear form $\left(R^\perp/R,\bar B_\varphi\right)$.

The following proposition was independently obtained by Burt Totaro.\footnote{Private communication}

\begin{prop}\label{prop: Vect^bil and FGor}
 For any $d\geq 2$, the maps
 \[\alpha\colon\Vect^\sym_{d-2}\to \FGor^{\o,0}_d\]
 sending $(V,B)$ to the oriented Gorenstein algebra $(R[x]/x^2\oplus V,\phi)$ of Example~\ref{ex:Gorenstein-algebras-from-bilinear-forms} and 
 \[\FGor^{\o,0}_d\to \Vect^\sym_{d-2}\]
 sending $(A,\phi)$ to $\left(R^\perp/R,\bar B_\varphi\right)$ are inverse up to $\A^1$-homotopy.
\end{prop}

\begin{proof}
 It is clear that the composition 
 \[\Vect^\sym_{d-2}\to \FGor^{\o,0}_d\to \Vect^\sym_{d-2}\]
 is isomorphic to the identity, as the orientation $\phi$ of the algebra $R[x]/x^2\oplus V$ satisfies $B_\phi((0,v),(0,v'))=\phi(B(v,v')x,0)=B(v,v')$ by definition. Following the proof of~\cite[Theorem~2.1]{robbery}, we will use the Rees algebra to construct an $\A^1$-homotopy of the other composition to identity. The difference with the proof in \textit{loc.\ cit.}\ is that we will use the orientation $\varphi$ to refine the filtration. Indeed, let us consider the natural filtration
 \[R\subset R^\perp\subset A\]
 of $A$. The corresponding Rees algebra is the $R[t]$-algebra given by
 \[\Rees(A,\varphi)=\{a\in A[t]\mid a_0\in R,\ \varphi(a_1)=0\}.\]
 By \cite[Lemma~2.2]{robbery}, it is a finite locally free $R[t]$-algebra such that
 \[\Rees(A,\varphi)/(t-1)\simeq A\quad\text{and}\quad \Rees(A,\varphi)/(t)\simeq R\oplus R^\perp/R\oplus A/R^\perp\,.\]
 We can equip $\Rees(A,\varphi)$ with an orientation $\tilde\varphi \colon \Rees(A,\varphi)\to R[t]$ given by $a\mapsto \frac{1}{t^2}\varphi(a)$ (note that $\varphi(a_0)=\varphi(a_1)=0$ by definition of the filtration, so $\varphi(a)$ is indeed divisible by $t^2$). We claim this gives $\Rees(A,\varphi)$ the structure of an oriented Gorenstein $R[t]$-algebra. Since this is a property local on the base, we can assume that $A$ is free as an $R$-module. Then let us choose a basis of $A$ of the form
 \[(1,e_1,\dots,e_{d-2},x)\,,\]
 where $1,e_1,\dots,e_{d-2}$ is a basis of $\ker \varphi$ and $x$ is such that $\varphi(x)=1$ and $\varphi(xe_i)=0$. Then the matrix of the bilinear form $B_\varphi$ is
 \[\begin{pmatrix} 0 & 0 & 1\\
                   0 & D & 0\\
                   1 & 0 & \varphi(x^2)\end{pmatrix},\]
 where $D$ is an invertible matrix (since $B_\varphi$ is non-degenerate). Then the collection
 \[(1,e_1t,\dots,e_{d-2}t,xt^2)\]
 is a basis for $\Rees(A,\varphi)$, and the matrix of the symmetric bilinear form $B_{\tilde \varphi}$ in this basis is
 \[\begin{pmatrix} 0 & 0 & 1\\
                   0 & D & 0\\
                   1 & 0 & \varphi(x^2)t^2\end{pmatrix},\]
which is invertible (since $D$ is). Hence, sending $(A,\varphi)$ to $(\Rees(A,\varphi),\tilde\varphi)$ defines a natural transformation
 \[\FGor^{\o,0}_d\to \FGor^{\o,0}_d(\A^1\times-)\,,\]
which provides the desired $\A^1$-homotopy. Indeed, at $t=0$ we have an isomorphism of $R$-algebras
\[
R\oplus R^\perp/R\oplus A/R^\perp \simeq R[x]/x^2\oplus R/R^\perp
\]
identifying $A/R^\perp$ with $Rx$ via $a\mapsto \phi(a)x$, under which the orientation $\tilde\phi_0=\phi\circ \pr_3$ of the left-hand side corresponds to the orientation of Example~\ref{ex:Gorenstein-algebras-from-bilinear-forms}, which extracts the coefficient of $x$.
\end{proof}

\begin{cor}
For any $d \geqslant 2$, there is a canonical $\A^1$-equivalence
\[\FGor^{\o,0}_d \simeq \GrO_{d-2}\]
on affine $\Z[\tfrac 12]$-schemes.
\end{cor}

\begin{proof}
Combine Propositions~\ref{prop:Vect^bil and Gr^o} and~\ref{prop: Vect^bil and FGor}.
\end{proof}

\section{The motivic hermitian K-theory spectrum}\label{sec:kq}

Fix a base scheme $S$.
By a \emph{presheaf with framed transfers} on $\Sm_S$ we mean a $\Sigma$-presheaf on the $\infty$-category $\Span^\fr(\Sm_S)$ constructed in \cite{deloop1} (i.e., a presheaf that takes finite coproducts in $\Span^\fr(\Sm_S)$ to products of spaces). We refer to \cite[\sectsign 3]{deloop1} for the definition of the $\infty$-categories $\H^\fr(S)$ and $\SH^\fr(S)$ of framed motivic spaces and framed motivic spectra over $S$.
The general reconstruction theorem \cite[Theorem 18]{framed-loc} states that the ``forget transfers'' functor implements an equivalence
\[
\SH^\fr(S)\simeq \SH(S)
\]
between the symmetric monoidal $\infty$-categories of framed motivic spectra and of motivic spectra. In particular, for any $E\in\SH(S)$, the presheaf of spaces $\Omega^\infty_\T E$ has a canonical structure of framed transfers.
We will abusively denote by $\Sigma^\infty_{\T,\fr}$ the composite functor
\[
\Pre_\Sigma(\Span^\fr(\Sm_S)) \to \H^\fr(S) \xrightarrow{\Sigma^\infty_{\T,\fr}} \SH^\fr(S)\simeq \SH(S).
\]

Let $\Span^{\fgor,\o}(\Sch_S)$ be the symmetric monoidal $(2,1)$-category of oriented finite Gorenstein correspondences: its objects are $S$-schemes and its morphisms are spans
\[
\begin{tikzcd}
   & Z \ar[swap]{ld}{f}\ar{rd}{g} & \\
  X &   & Y
\end{tikzcd}
\]
where $Z\in\FGor^\o(X)$, that is, $f$ is finite locally free and equipped with a trivialization $\omega_f\simeq \sO_Z$ (see Definition~\ref{dfn:oriented-Gorenstein-algebra}). The symmetric monoidal structure is given by the product of $S$-schemes (for two finite locally free morphisms $f\colon Z\to X$ and $f'\colon Z'\to X'$ over $S$, the dualizing sheaf $\omega_{f\times_S f'}$ is the external tensor product $\omega_f\boxtimes_S \omega_{f'}$, so trivializations of $\omega_f$ and $\omega_{f'}$ induce a trivialization of $\omega_{f\times_Sf'}$). The wide subcategory where $f$ is finite syntomic is the $(2,1)$-category $\Span^{\fsyn,\o}(\Sch_S)$ of oriented finite syntomic correspondences considered in \cite[\sectsign 4.2]{deloop3}. We write $\Span^{\fgor,\o}(\Sm_S)$ for the full subcategory spanned by the smooth $S$-schemes.

The presheaf on $\Span^{\fgor,\o}(\Sch_S)$ represented by $S$ is the presheaf of groupoids $\FGor^\o$, which is therefore a commutative monoid in $\Pre_\Sigma(\Span^{\fgor,\o}(\Sch_S))$ (with respect to the Day convolution).
We observe that the presheaf of groupoids $\Vect^\sym$ can also be promoted to a commutative monoid object in $\Pre_\Sigma(\Span^{\fgor,\o}(\Sch_S))$. Given a finite Gorenstein morphism $f\colon Z\to X$ with orientation $\phi\colon f_*\sO_Z\to\sO_X$, we have a pushforward functor
\[
f_*\colon \Vect^\sym(Z) \to \Vect^\sym(X), \quad (\sE,B)\mapsto (f_*\sE, \phi\circ f_*B).
\]
This pushforward is compatible with composition, base change, and the tensor product (it satisfies a projection formula) up to canonical isomorphisms. Since $\Vect^\sym$ is an ordinary groupoid, it is straightforward to check the coherence of these canonical isomorphisms, which gives $\Vect^\sym$ the structure of a commutative monoid in $\Pre_\Sigma(\Span^{\fgor,\o}(\Sch_S))$. 
Finally, the forgetful map $\eta\colon \FGor^\o\to \Vect^\sym$ can be uniquely promoted to a morphism of commutative monoids in $\Pre_\Sigma(\Span^{\fgor,\o}(\Sch_S))$, since $\FGor^\o$ is the initial such object. We will write $\FGor^\o_S$ and $\Vect^\sym_S$ for the restrictions of these presheaves to smooth $S$-schemes.

By Lemma~\ref{lem:det(L)} below (see also \cite[Theorem A.1]{BachmannWickelgren}), the determinant $\det\colon K_{\rk=0}\to\Pic$ defines a symmetric monoidal forgetful functor
\[
\Span^\fr(\Sch_S) \to \Span^{\fsyn,\o}(\Sch_S)\subset \Span^{\fgor,\o}(\Sch_S).
\]
We can therefore regard $\FGor^\o$ and $\Vect^\sym$ as commutative monoids in $\Pre_\Sigma(\Span^\fr(\Sch_S))$.

\begin{lem}\label{lem:det(L)}
Let $f$ be a finite syntomic map, $L_f$ its cotangent complex, and $\omega_f$ its relative dualizing sheaf. Then there is a canonical isomorphism
\[\det(L_f) \simeq \omega_f .\]
Moreover, this isomorphism is compatible with base change and composition in the obvious way.
\end{lem}

\begin{proof}
By \cite[Tag 0FKD]{stacks}, there is such a family of isomorphisms for $f$ a finite syntomic morphism between noetherian schemes. Since they are compatible with base change, they uniquely determine the desired isomorphisms for $f$ between qcqs schemes (by noetherian approximation), whence between arbitrary schemes by descent.
\end{proof}

\begin{rem}
	Lemma~\ref{lem:det(L)} holds more generally if $f$ is any quasi-smooth morphism of derived schemes, with $\omega_f=f^!(\sO)$, but the proof is more complicated (in fact, the functor $f^!$ seems not to have been defined in this generality, although one can also define $\omega_f$ more directly).
\end{rem}

In \cite{Schlichting3}, Schlichting defined the Grothendieck--Witt spectrum of a qcqs $\Z[\tfrac 12]$-scheme $S$; we shall denote by $\GWspace(S)$ its underlying space.\footnote{In \emph{loc.\ cit.}, Schlichting works with schemes admitting an ample family of line bundles, but remarks that this assumption can be removed if one uses perfect complexes instead of strictly perfect complexes. In any case, Schlichting's setting suffices for our purposes, since it suffices to define the motivic localization of $\GWspace$.}
The presheaf $\GWspace$ is then a Nisnevich sheaf \cite[Theorem 9.6]{Schlichting3}, and it is $\A^1$-invariant on regular schemes \cite[Theorem 9.8]{Schlichting3}.
By the example in~\cite[p.~1781]{Schlichting3} and by~\cite[Remark~4.13]{Schlichting1}, for $R$ a $\Z[\tfrac 12]$-algebra there is a natural equivalence
\begin{equation*}\label{eqn:GW-Vect-sym}
	\GWspace(R) \simeq \Vect^\sym(R)^\gp .
\end{equation*}
In particular, we have an equivalence $\GWspace \simeq L_\Zar\Vect^{\sym,\gp}$ making $\GWspace$ into a presheaf of $\Einfty$-ring spaces.
By \cite[Remark~5.9 and Theorem~9.10]{Schlichting3}, we have a chain of equivalences of $\GWspace$-modules:
\[\GWspace\simeq \GWspace^{[-4]}\simeq \Omega^4_{\P^1}\GWspace.\]
The element $1$ in the left-hand side gives a canonical element 
$\betah\colon (\P^1)^{\wedge 4} \to \GWspace$, called the \emph{hermitian Bott element}.
The above equivalence means that $\GWspace$ is $\betah$-periodic in the sense of \cite[Section 3]{HoyoisCdh}.

Let $S$ be a scheme with $2\in \sO(S)^\times$.
We define the motivic hermitian K-theory spectrum $\KQ_S\in\SH(S)$ as the motivic spectrum associated with the $\T^{\wedge 4}$-prespectrum
\[
(L_\mot\GWspace_S, L_\mot\GWspace_S,\dotsc),
\]
where $\GWspace_S$ is the restriction of $\GWspace$ to $\Sm_S$, with structure maps $L_\mot\GWspace_S \to \Omega_\T^4 L_\mot\GWspace_S$ induced by the hermitian Bott element. Since the structure maps are $\GWspace_S$-module maps, we can regard $\KQ_S$ as an object of $\Mod_{\GWspace}(\SH(S))$. 
We write $\kq_S$ for the very effective cover of $\KQ_S$.

Note that when $S$ is regular, then $\GWspace_S\simeq L_\mot\GWspace_S$ and the above prespectrum is already a spectrum; in that case, $\KQ_S$ was originally defined by Hornbostel \cite{Hornbostel:2005}. We refer to Proposition~\ref{prop:KQ=GW} below for a justification of our definition in general.

\begin{rem}\label{rem:2}
There is work in progress by B.~Calmès, Y.~Harpaz and the third author on constructing the motivic hermitian K-theory spectrum over base schemes where $2$ need not be invertible \cite{CalmesNardinHarpaz}. All of the results below remain valid without that assumption (in Proposition~\ref{prop:KSp}(2), one must however use skew-symmetric instead of alternating forms).
\end{rem}

\begin{lem}\label{lem:KQ-is-E-infinity}
	Let $S$ be a regular $\Z[\tfrac 12]$-scheme. Then the object $\KQ_S\in \Mod_{\GWspace}(\SH(S))$ has a unique structure of $\Einfty$-algebra lifting the $\Einfty$-ring structure of $\GWspace_S$.
\end{lem}

\begin{proof}
	Since $S$ is regular we have $\KQ_S=(\GWspace_S,\GWspace_S,\dotsc)$.
	The action of $\betah$ on $\KQ_S$ is given levelwise by its action on $\GWspace_S$, and hence it is invertible. The motivic spectrum $\KQ_S$ is thus $\betah$-periodic in the sense of \cite[Section 3]{HoyoisCdh}. By \cite[Proposition 3.2]{HoyoisCdh}, there is an equivalence of symmetric monoidal $\infty$-categories
	\[
	\Sigma^\infty_\T(-)[\betah^{-1}]: P_\betah\Mod_{\GWspace}(\H_*(S)) \rightleftarrows P_\betah \Mod_{\GWspace}(\SH(S)): \Omega^\infty_\T,
	\]
	where $P_\betah\Mod_{\GWspace}$ on either side denotes the subcategory of $\betah$-periodic $\GWspace$-modules.
	Since $\Omega^\infty_\T \KQ_S\simeq \GWspace_S$, this proves the claim.
\end{proof}

By construction, there is a $\T^{\wedge 4}$-prespectrum of presheaves on \emph{all} $\Z[\tfrac 12]$-schemes whose restriction to $\Sm_S$ gives $\KQ_S$ for all $S$. This implies that $S\mapsto \KQ_S$ is a section of the cocartesian fibration over $\Sch_{\Z[\frac 12]}^\op$ classified by $\SH(-)$. In particular, for any morphism of $\Z[\tfrac 12]$-schemes $f\colon T\to S$, we have a canonical comparison map $f^*(\KQ_S)\to \KQ_T$ in $\SH(T)$.

\begin{lem}\label{lem:KQ-bc}
	For any morphism of $\Z[\tfrac 12]$-schemes $f\colon T\to S$, the canonical map $f^*(\KQ_S)\to \KQ_T$ is an equivalence.
\end{lem}

\begin{proof}
	The functor $f^*$ is compatible with spectrification,  motivic localization, and group completion. Hence, it suffices to show that the canonical map $f^*(\Vect^\sym_S)\to \Vect^\sym_T$ is a motivic equivalence of presheaves on $\Sm_T$. In fact, it is a Zariski-local equivalence by \cite[Proposition A.0.4]{deloop3}, since $\Vect^\sym$ is a smooth algebraic stack with affine diagonal.
\end{proof}

Combining Lemmas~\ref{lem:KQ-is-E-infinity} and \ref{lem:KQ-bc}, we obtain a canonical $\Einfty$-ring structure on $\KQ_S$ for any $\Z[\tfrac 12]$-scheme $S$, which is compatible with base change (an $\Einfty$-ring structure on $\KQ_S$ was also constructed in this generality by Lopez-Avila in \cite{AvilaThesis}; we suspect but do now know that it is equivalent to ours).

Next, we show that the motivic spectrum $\KQ_S$ represents Karoubi–Grothendieck–Witt theory made $\A^1$-homotopy invariant:

\begin{prop}\label{prop:KQ=GW}
	Let $S$ be a qcqs scheme with $2\in\sO(S)^\times$ and let $n\in\Z$. Then there is a natural equivalence of spectra (and of $\Einfty$-ring spectra if $n=0$)
	\[
	\Gamma(S,\Sigma^n_\T\KQ) \simeq (\Lhtp\KGW^{[n]})(S),
	\]
	where $\KGW^{[n]}$ is the $n$th shifted Karoubi–Grothendieck–Witt spectrum \cite[Definition 8.6]{Schlichting3}.
\end{prop}

\begin{proof}
	We start with the observation that the restriction of the Bott element $\beta\in \GWspace^{[1]}(\P^1)$ defined in \cite[\sectsign 9.4]{Schlichting3} to either $\A^1$ or $\P^1-0$ is canonically null-homotopic. This gives rise to the following four variations on the domain of the Bott element, which are all motivically equivalent:
	\[
	\begin{tikzcd}
		\T=\A^1/\G_m \ar{r} \ar{d} & \P^1/(\P^1-0) \ar{dr}{\beta} & \P^1/\infty \ar{l} \ar{d}{\beta} \\
		\Sigma\G_m \ar{rr}{\beta} & & \GWspace^{[1]}\rlap.
	\end{tikzcd}
	\]
	Let us write $\GWspectrum$ for the Grothendieck–Witt spectrum and $\GWspectrum_{\geq 0}$ for its connective cover. We consider the following four ``periodic'' prespectra with structure maps given by $\betah=\beta^4$:
	\begin{itemize}
		\item[(A)] $(\GWspace,\GWspace,\dotsc)$,
		\item[(B)] $(\GWspectrum_{\geq 0},\GWspectrum_{\geq 0},\dotsc)$,
		\item[(C)] $(\GWspectrum,\GWspectrum,\dotsc)$,
		\item[(D)] $(\KGW,\KGW,\dotsc)$.
	\end{itemize}
	Here, (A) is a prespectrum in presheaves of pointed spaces, which defines $\KQ$, whereas (B)–(D) are prespectra in presheaves of spectra. Note also that (A), (C), and (D) are $(\P^1)^{\wedge 4}$-spectra, by the projective bundle formula \cite[Theorem~9.10]{Schlichting3}.	
	Since $\KGW^{[n]}$ is a Nisnevich sheaf of spectra on qcqs schemes \cite[Theorem 9.6]{Schlichting3}, its motivic localization is given by $\Lhtp\KGW^{[n]}$. Moreover, since $\Lhtp$ commutes with $\Omega_{\P^1}$ and $\Omega_{\P^1}\KGW^{[n]}\simeq \KGW^{[n-1]}$, we see that the motivic spectrum defined by (D) satisfies the desired conclusion. It will thus suffice to show that (A)–(D) all define the same motivic spectrum. The fact that we get an equivalence of $\Einfty$-rings when $n=0$ follows by repeating the proof of Lemma~\ref{lem:KQ-is-E-infinity} with $\Lhtp\KGW$ instead of $\GWspace$.
	
	 The maps $(B)\to (C)\to(D)$ become equivalences after $(\Sigma\G_m)^{\wedge 4}$-spectrification, since the maps \[\Omega_{\Sigma\G_m}^{n}\GWspectrum_{\geq 0}\to \Omega_{\Sigma\G_m}^n\GWspectrum\to \Omega_{\Sigma\G_m}^n\KGW\] induce isomorphisms on $\pi_i$ for $i\geq -n$ \cite[Proposition 9.3(1)]{Schlichting3}. Hence, the motivic spectra associated with (B), (C), and (D) are the same. Furthermore, since $\Lhtp$ commutes with spectrification, we see that the motivic spectrum associated with (B) can be obtained in two steps: first apply $\Lhtp$ and then $(\Sigma\G_m)^{\wedge 4}$-spectrify. As $\Omega^\infty$ preserves sifted colimits of \emph{connective} spectra, this explicit formula for the localization of (B) implies that the motivic spectra associated with (A) and (B) are identified under the equivalence between $\T^{\wedge 4}$-spectra in motivic spaces and in motivic $S^1$-spectra (which is implemented by the functor $\Omega^\infty$ levelwise). This completes the proof.
\end{proof}

\begin{samepage}
\begin{prop}\label{prop:KQ-kq-Vect^bil}
	Assume $2\in\sO(S)^\times$.
	\begin{enumerate}
		\item There is an equivalence of $\Einfty$-ring spectra
	\[
	\KQ_S\simeq (\Sigma^\infty_{\T,\fr}\Vect^\sym_S)[\betah^{-1}].
	\]
	\item If $S$ is regular over a field, there is an equivalence of $\Einfty$-ring spectra
	\[
	\kq_S\simeq \Sigma^\infty_{\T,\fr}\Vect^\sym_S.
	\]
	\end{enumerate}
\end{prop}
\end{samepage}

\begin{proof}
	Since $\GWspace$ is a Nisnevich sheaf, which is given by $\Vect^{\sym,\gp}$ on affine schemes, we have $\GWspace=L_\Nis\Vect^{\sym,\gp}$. Since Nisnevich sheafification preserves presheaves with oriented finite Gorenstein transfers (cf.\ \cite[Proposition 3.2.4]{deloop1}), the $\T^{\wedge 4}$-prespectrum defining $\KQ_S$ is a $\GWspace_S$-module with oriented finite Gorenstein transfers. In particular, it defines a $\GWspace_S$-module in framed motivic spectra whose underlying motivic spectrum is $\KQ_S$. Repeating the proof of Lemma~\ref{lem:KQ-is-E-infinity} in the framed setting, we see that $\KQ_S \simeq (\Sigma^\infty_{\T,\fr}\Vect^\sym_S)[\betah^{-1}]$ as $\Einfty$-rings when $S$ is regular. By Lemma~\ref{lem:KQ-bc}, the left-hand side is stable under base change. The right-hand side is stable under base change as well by \cite[Lemma~16]{framed-loc}, which proves (1) in general.
	The proof of (2) uses the motivic recognition principle and is identical to the proof of \cite[Corollary 5.2]{robbery}.
\end{proof}

We have analogous results with the shifted Grothendieck–Witt space $\GWspace^{[2]}$.
For a $\Z[\tfrac 12]$-scheme $S$, we define $\KSp_S\in\Mod_\GWspace(\SH(S))$ to be the motivic spectrum associated with the $\T^{\wedge 4}$-prespectrum
\[
(L_\mot\GWspace^{[2]}, L_\mot\GWspace^{[2]},\dotsc),
\]
with structure maps induced by $\betah$. We write $\ksp_S$ for the very effective cover of $\KSp_S$.

For a $\Z[\tfrac 12]$-algebra $R$, we have 
\[
\GWspace^{[2]}(R)\simeq \Vect^\alt(R)^\gp
\]
where $\Vect^\alt$ is the stack of non-degenerate alternating bilinear forms. Since the latter is a smooth algebraic stack with affine diagonal, we deduce as in Lemma~\ref{lem:KQ-bc} that $\KSp_S$ is stable under arbitrary base change. From the equivalence $\Omega_{\P^1}^2\GWspace^{[2]}\simeq \GWspace$, it follows immediately that
\[\KSp_S\simeq \Sigma^2_\T\KQ_S\]
when $S$ is regular, whence in general by base change. 

Arguing exactly as in Proposition~\ref{prop:KQ-kq-Vect^bil}, we obtain the following result:

\begin{prop}\label{prop:KSp}
	Assume $2\in\sO(S)^\times$.
	\begin{enumerate}
		\item There is an equivalence of $\KQ_S$-module spectra
	\[
	\KSp_S\simeq (\Sigma^\infty_{\T,\fr}\Vect^\alt_S)[\betah^{-1}].
	\]
	\item If $S$ is regular over a field, there is an equivalence of $\kq_S$-module spectra
	\[
	\ksp_S\simeq \Sigma^\infty_{\T,\fr}\Vect^\alt_S.
	\]
	\end{enumerate}
\end{prop}
 
\begin{rem}
	 Combining Propositions~\ref{prop:KQ-kq-Vect^bil} and \ref{prop:KSp}, we find
	\[
	\KQ_S\oplus \KSp_S \simeq \Sigma^\infty_{\T,\fr}\big(\Vect^\sym_S\times \Vect^\alt_S\big)[\betah^{-1}].
	\]
	One can easily check that the groupoid $\Vect^\sym\times \Vect^\alt$ has a symmetric monoidal structure given by tensoring bilinear forms, which is compatible with the oriented finite Gorenstein transfers. We deduce that $\KQ_S\oplus \KSp_S$ has an $\Einfty$-ring structure such that the summand inclusion $\KQ_S\to \KQ_S\oplus \KSp_S$ is an $\Einfty$-map.
\end{rem}

\begin{rem}\label{rem:SL-orientation}
Recall that $\MSL_S\simeq \Sigma^\infty_{\T,\fr}\FSyn^\o_S$ \cite[Theorem 3.4.3]{deloop3}. Hence, by Proposition~\ref{prop:KQ-kq-Vect^bil}, the forgetful map $\FSyn^\o\to \Vect^\sym$ induces a morphism of $\Einfty$-ring spectra
\[
\MSL_S \to \KQ_S
\]
for any $\Z[\tfrac 12]$-scheme $S$, which factors through $\kq_S$ since $\MSL_S$ is very effective.
\end{rem}

\begin{rem}
	The oriented finite Gorenstein transfers in hermitian K-theory are known to not fully depend on the trivialization of the dualizing line $\omega_f$, but only on a choice of square root $\omega_f\simeq \mathcal L^{\otimes 2}$. Correspondingly, the coherent $\SL$-orientation of $\KQ_S$ from Remark~\ref{rem:SL-orientation} can be improved to an $\SL^c$-orientation.
	Define the presheaf $K^{\SL^c}\colon \Sch^\op\to \Spc$ by the pullback square
	\[
	\begin{tikzcd}
		K^{\SL^c} \ar{r} \ar{d} & K \ar{d}{\det} \\
		\Pic \ar{r}{2} & \Pic\rlap,
	\end{tikzcd}
	\]
	where $\Pic(X)$ is the groupoid of invertible sheaves on $X$.
	The motivic $\Einfty$-ring spectrum $\MSL^c$ is the Thom spectrum associated with the composition \[K^{\SL^c}\times_{\uZ} \{0\}\to K\to \Pic(\SH),\] in the sense of \cite[Section 16]{norms}. By \cite[Theorem 3.3.10]{deloop3}, there is an equivalence of $\Einfty$-rings
	\[
	\MSL^c\simeq \Sigma^\infty_{\T,\fr}\FSyn^{c}
	\]
	where $\FSyn^c(X)$ is the groupoid of triples $(Z,\sL,\lambda)$ with $Z$ a finite syntomic $X$-scheme, $\sL$ an invertible sheaf on $Z$, and $\lambda\colon \omega_{Z/X}\simeq \sL^{\otimes 2}$. There is a map $\FSyn^c\to\Vect^\sym$ sending $(Z,\sL,\lambda)$ to $f_*(\sL)$, where $f\colon Z\to X$ is the structure map, with the symmetric bilinear form
	\[
	f_*(\sL) \otimes f_*(\sL) \to f_*(\sL^{\otimes 2}) \stackrel\lambda\simeq f_*(\omega_{Z/X}) \to\sO_X.
	\]
	This induces an $\Einfty$-$\SL^c$-orientation $\MSL^c\to\KQ$ refining the $\SL$-orientation of Remark~\ref{rem:SL-orientation}.
\end{rem}

Combining Proposition~\ref{prop:KQ-kq-Vect^bil} and Theorem~\ref{thm:main-gp}, we obtain a description of the motivic hermitian K-theory spectrum in terms of oriented Gorenstein algebras:

\begin{thm}\label{thm:KQ-kq-FGor}
	Assume $2\in\sO(S)^\times$.
	\begin{enumerate}
		\item There is an equivalence of $\Einfty$-ring spectra 
\[\KQ_S\simeq (\Sigma^\infty_{\T,\fr}\FGor^\o_S)[\betah^{-1}],\]
where $\betah\in\pi_{8,4}\Sigma^\infty_{\T,\fr}\FGor^\o_S$ is transported through the equivalence of Theorem~\ref{thm:main-gp}.
\item If $S$ is regular over a field, there is an equivalence of $\Einfty$-ring spectra
\[
\kq_S \simeq \Sigma^\infty_{\T,\fr} \FGor^\o_S.
\]
	\end{enumerate}
\end{thm}

\section{The Milnor--Witt motivic cohomology spectrum}\label{sec:MW}

Let $S$ be a scheme that is pro-smooth over a field or over a Dedekind scheme (i.e., $S$ is locally a cofiltered limit of smooth schemes; for example, by Popescu's theorem~\cite[Tag 07GC]{stacks}, $S$ is regular over a field). The \emph{Milnor–Witt motivic cohomology spectrum} $\hzmw_S\in\SH(S)$ is defined by 
\[\hzmw_S=\underline\pi^\eff_0(\MonUnit_S),\]
where $\underline\pi^\eff_*$ are the homotopy groups in the effective homotopy $t$-structure on $\SH^\eff(S)$ (whose connective part is the subcategory $\SH^\veff(S)$ of very effective motivic spectra).
This definition is due to Bachmann \cite{BachmannSlices,BachmannDedekind}, and it is known to be equivalent to the original definition of Calmès and Fasel over a perfect field of characteristic not $2$~\cite{BachmannFasel}.
Note that $\hzmw_S$ is an $\Einfty$-ring spectrum (even a normed spectrum, see \cite[Proposition 13.3]{norms}) that is stable under pro-smooth base change.

For $S$ as above, we define a presheaf of rings $\uGW_S$ on $\Sm_S$ in two cases:
\begin{enumerate}
	\item if $S$ is over a perfect field $k$, $\uGW_S$ is the Zariski sheafification of the left Kan extension of the sheaf of unramified Grothendieck–Witt rings over $\Sm_k$ defined by Morel \cite[\sectsign 3.2]{Morel};
	\item if $2\in\sO(S)^\times$, $\uGW_S$ is the Nisnevich sheafification of the presheaf of Grothendieck–Witt rings, i.e., $\uGW_S=L_\Nis\pi_0\Vect_S^{\sym,\gp}$.
\end{enumerate}
Note that these two definitions agree when they both apply (see the proof of \cite[Theorem 10.12]{norms}).
We can promote $\uGW_S$ to a commutative monoid in presheaves with framed transfers as follows.
In case (1), since $\uGW=(\underline{K}_3^{MW})_{-3}$, \cite[Theorem 5.19]{BachmannYakerson} implies that $\uGW_k$ admits a unique structure of presheaf with framed transfers compatible with its $\uGW_k$-module structure and extending Morel's transfers for monogenic field extensions \cite[\sectsign 4.2]{Morel}. In case (2), $\uGW_S$ inherits oriented finite Gorenstein transfers from $\Vect^\sym$. The uniqueness of the framed transfers in (1) implies that they agree with those in (2) when $2$ is invertible.

\begin{rem}
	Over a field of characteristic $2$, there is a canonical epimorphism of Nisnevich sheaves $L_\Nis\pi_0\Vect^{\sym,\gp}\to\uGW$, but we do not know if it is an isomorphism (see \cite[Remark 4.7]{deloop4}).
\end{rem}

\begin{lem}\label{lem:hzmw-transfers}
	Let $S$ be pro-smooth over a field or a Dedekind scheme. If $S$ has mixed characteristic, assume also that $2\in\sO(S)^\times$. Then there is an equivalence $\Omega^\infty_{\T,\fr}\hzmw_S\simeq \uGW_S$ of commutative monoids in $\Pre_\Sigma(\Span^\fr(\Sm_S))$.
\end{lem}

\begin{proof}
	In the equicharacteristic case, it suffices to prove the result over a perfect field. 
	By \cite[Proposition~4(3)]{BachmannSlices}, $\hzmw$ is the effective cover of $\underline\pi_0(\1)_*$, and hence we have isomorphisms of rings $\Omega^\infty_\T\hzmw\simeq\underline\pi_0(\1)_0\simeq \uGW$. It remains to compare the framed transfers on either side. By \cite[Corollary 5.17]{BachmannYakerson}, it suffices to compare the transfers induced by a monogenic field extension $K\subset L$ with chosen generator $a\in L$. This was done in the proof of \cite[Proposition 4.3.17]{deloop2}.
	
	Assume now that $2\in\sO(S)^\times$. By \cite[Definition 4.1 and Corollary 4.9]{BachmannDedekind}, we have an isomorphism $\Omega^\infty_{\T}\hzmw_S\simeq \uGW_S$ such that the following square commutes, where $\kq^\fr_S=\Sigma^\infty_{\T,\fr}\Vect^{\sym,\gp}_S$:
	\[
	\begin{tikzcd}
		\Vect^{\sym,\gp}_S \ar{r}{\mathrm{unit}} \ar{d} & \Omega^\infty_{\T}\kq^\fr_S \ar{d} \\
		\uGW_S \ar{r}{\simeq} & \Omega^\infty_{\T}\hzmw_S\rlap.
	\end{tikzcd}
	\]
	The top and left arrows are morphisms of presheaves with framed transfers by definition, and so is the right vertical arrow since it is $\Omega^\infty_{\T}$ of the morphism $\kq_S^\fr\to \underline\pi^\eff_0\kq_S^\fr\simeq \hzmw_S$. Using the compatibility of the Nisnevich topology with framed transfers \cite[Proposition 3.2.14]{deloop1}, we deduce that the bottom isomorphism is compatible with framed transfers.
\end{proof}

The next theorem is the analogue of \cite[Theorem 21]{framed-loc} for Milnor–Witt motivic cohomology. We are grateful to Tom Bachmann for providing the rigidity argument in the mixed characteristic case.

\begin{thm}\label{thm:HtildeZ}
Let $S$ be pro-smooth over a field or a Dedekind scheme. If $S$ has mixed characteristic, assume also that $2\in\sO(S)^\times$. Then there is an equivalence of motivic $\Einfty$-ring spectra $\hzmw_S \simeq  \Sigma^\infty_{\T,\fr} \uGW_S$ in $\SH(S)$.
\end{thm}

\begin{proof}
	The equivalence $\Omega^\infty_{\T,\fr}\hzmw_S\simeq \uGW_S$ of Lemma~\ref{lem:hzmw-transfers} induces a canonical $\Einfty$-map
	\[
	\Sigma^\infty_{\T,\fr}\uGW_S\to \hzmw_S
	\]
	in $\SH(S)$. Since both sides are compatible with pro-smooth base change, it is enough to show that it is an equivalence when $S$ is a perfect field or a Dedekind domain. In the former case, the result follows directly from the motivic recognition principle \cite[Theorem~3.5.14]{deloop1}, since $\hzmw_S$ is very effective. 
	Let us therefore assume that $S$ is a Dedekind domain with $2\in\sO(S)^\times$. By \cite[Proposition B.3]{norms}, it suffices to show that both $\Sigma^\infty_{\T,\fr}\uGW_S$ and $\hzmw_S$ are stable under base change to the residue fields. For $\hzmw_S$, this holds by \cite[Theorem 4.4]{BachmannDedekind}. For $\Sigma^\infty_{\T,\fr}\uGW_S$, in view of \cite[Lemma 16]{framed-loc}, it suffices to show that the canonical map $s^*(\uGW_S)\to \uGW_{\kappa(s)}$ in $\Pre(\Sm_{\kappa(s)})$ is a motivic equivalence for all $s\in S$.
	Since $\uGW\simeq\uW\times_{\uZ/2}\uZ$ and Witt groups satisfy rigidity for henselian local $\Z[\tfrac 12]$-algebras \cite[Lemma 4.1]{Jacobson}, this follows from Lemma~\ref{lem:rigidity} below. 
\end{proof}

\begin{rem}
	The proof of Theorem~\ref{thm:HtildeZ} shows that $\hzmw_S \simeq  \Sigma^\infty_{\T,\fr} \uGW_S$ for any $\Z[\tfrac 12]$-scheme $S$, if one defines $\hzmw_S$ by base change from $\Spec\Z$ and $\uGW_S$ as $L_\Nis\pi_0\Vect^{\sym,\gp}_S$.
\end{rem}

We say that a presheaf $\sF\in \Pre(\Sch_S)$ is \emph{finitary} if it transforms limits of cofiltered diagrams of qcqs schemes with affine transition maps into colimits. We say that $\sF$ \emph{satisfies rigidity} for a henselian pair $(A,I)$ if the map $\sF(A)\to \sF(A/I)$ is an equivalence.

\begin{lem}[Bachmann]
	\label{lem:rigidity}
	Let $\sF\in\Pre(\Sch_S)$ be a finitary presheaf satisfying rigidity for all pairs $(A,I)$ where $A$ is an essentially smooth henselian local ring over $S$.
	Suppose either that $S$ is locally of finite Krull dimension or that $\sF$ is truncated.
	For $X\in\Sch_S$, denote by $\sF_X$ the restriction of $\sF$ to $\Sm_X$.
	Then, for every morphism $f\colon Y\to X$ in $\Sch_S$, the canonical map
	\[
	f^*(\sF_X) \to \sF_Y
	\]
	in $\Pre(\Sm_Y)$ is a motivic equivalence.
\end{lem}

\begin{proof}
	Note that the statement depends only on the Nisnevich sheafification of $\sF$, so we are free to assume that $\sF(\emptyset)=*$.
	By 2-out-of-3, we can assume $X=S$.	Since the question is local on $Y$ and $\sF$ is finitary, we can further assume that $Y$ is a closed subscheme of $\A^n_S$. Replacing $S$ by $\A^n_S$, we see that it suffices to prove the result for a closed immersion $i\colon Z\hook S$. Let $j\colon U\hook S$ be its open complement and consider the commutative square
	\[
	\begin{tikzcd}
		j_\sharp\sF_U \ar{r} \ar{d} & \sF_S \ar{d} \\
		U \ar{r} & i_*\sF_Z
	\end{tikzcd}
	\]
	in $\Pre(\Sm_S)$. We claim that this square is a Nisnevich-local pushout square.
	If $\sF$ is truncated, this is a square of truncated presheaves, so the claim can be checked on stalks; the same is true if $S$ is locally of finite dimension by \cite[Corollary 3.27]{ClausenMathew}. Using that $\sF$ is finitary, the stalks over an essentially smooth henselian local scheme $X$ are given by
	\[
	\begin{tikzcd}
		\sF(X) \ar{r} \ar{d} & \sF(X) \ar{d} \\
		* \ar{r} & \sF(\emptyset)
	\end{tikzcd}
	\quad\text{or}\quad
	\begin{tikzcd}
		\emptyset \ar{r} \ar{d} & \sF(X) \ar{d} \\
		\emptyset \ar{r} & \sF(X_Z)\rlap,
	\end{tikzcd}
	\]
	depending on whether $X_Z$ is empty or not. Since $\sF$ is rigid and satisfies $\sF(\emptyset)=*$, this proves the claim. On the other hand, by the Morel–Voevodsky localization theorem, the square
	\[
	\begin{tikzcd}
		j_\sharp\sF_U \ar{r} \ar{d} & \sF_S \ar{d} \\
		U \ar{r} & i_*i^*\sF_S
	\end{tikzcd}
	\]
	is motivically cocartesian (see \cite[\S3 Theorem 2.21]{MV} or \cite[Corollary 5]{framed-loc}). It follows that the canonical map $i_*i^*\sF_S\to i_*\sF_Z$ is a motivic equivalence. Since $i_*$ commutes with $L_\mot$ for presheaves satisfying $\sF(\emptyset)=*$ and since $i_*\colon \H(Z)\to \H(S)$ is fully faithful, we deduce that $i^*\sF_S\to\sF_Z$ is a motivic equivalence, as desired.
\end{proof}

\section{Modules over hermitian K-theory}

We begin with a straightforward adaptation of Bachmann's cancellation theorem for finite flat correspondences \cite{BachmannFFlatCancellation} to Gorenstein and oriented Gorenstein correspondences.

Let $k$ be a perfect field and $\sC$ be a motivic $\infty$-category of correspondences over $k$, in the sense of \cite[Definition 2.1]{BachmannFFlatCancellation} (our primary interest is the example $\sC = \Span^{\fgor,\o}(\Sm_k)$). 
We denote by $\h^\sC(X)\in \Pre_\Sigma(\sC)$ the presheaf represented by a smooth $k$-scheme $X$, by $\H^\sC(k)\subset \Pre_\Sigma(\sC)$ the full subcategory of $\A^1$-invariant Nisnevich sheaves, and by $\SH^\sC(k)$ the $\infty$-category of $\T$-spectra in $\H^\sC(k)_*$. 

Recall that:
\begin{itemize}
	\item $\sC$ \emph{satisfies cancellation} if the functor $\h^\sC(\G_m,1)\otimes(-)\colon \H^\sC(k)^\gp\to \H^\sC(k)^\gp$ is fully faithful;
	\item $\sC$ \emph{satisfies rational contractibility} if the presheaf $\h^\sC((\G_m,1)^{\wedge n})^\gp$ on $\Sm_k$ is rationally contractible for all $n\geq 1$.
\end{itemize}

The following result is only new in the cases of Gorenstein and oriented Gorenstein correspondences, but the same proof works in all cases.

\begin{prop}\label{prop: cancellation + rat contract}
	Let $k$ be a perfect field. Let $\sC$ be the $\infty$-category of smooth $k$-schemes and correspondences of one of the following types:
	\begin{enumerate}
		\item finite flat;
		\item finite Gorenstein;
		\item finite syntomic;
		\item oriented finite Gorenstein;
		\item oriented finite syntomic;
		\item framed finite syntomic.
	\end{enumerate}
	Then $\sC$ satisfies cancellation and rational contractibility.
\end{prop}

\begin{proof}
	The case of finite flat correspondences is \cite[Theorem 3.5 and Proposition 3.8]{BachmannFFlatCancellation}. Essentially the same argument applies to all the other cases.
	
	For cancellation, we use the criterion \cite[Proposition 2.16]{BachmannFFlatCancellation} with $G=\h^\sC(\G_m)$. 
	The object $\h^\sC(\G_m,1)$ is symmetric in $\H^\sC(k)$, because $\H^\sC(k)$ is prestable \cite[Lemma 2.10]{BachmannFFlatCancellation} and $\Sigma(\G_m,1)$ is already symmetric in $\H(k)_*$.
	We then need to construct, for each $Y\in\Sm_k$, a map
	\[
	\rho\colon \Hom(\G_m,L_\mot\h^\sC(\G_m\times Y)^\gp) \to L_\mot\h^\sC(Y)^\gp.
	\]
	in $\H^\sC(k)^\gp$ satisfying some conditions.
	
	For $m,n\geq 0$, we define
	\begin{align*}
		g_m^+\colon \G_m\times\G_m\to\A^1,& \quad g_m^+(t_1,t_2)=t_1^m+1,\\
		g_m^-\colon \G_m\times\G_m\to\A^1,& \quad g_m^+(t_1,t_2)=t_1^m+t_2,\\
		h_{m,n}^\pm\colon \A^1\times\G_m\times\G_m\to\A^1,&\quad h_{m,n}^\pm(s,t_1,t_2) = sg_n^\pm(t_1,t_2)+(1-s)g_m^\pm(t_1,t_2).
	\end{align*}
	(Thus, $h_{m,n}^\pm$ is the straight-line homotopy from $g_m^\pm$ to $g_n^\pm$.)
	Given a span $\G_m\times X\leftarrow Z\to \G_m\times Y$ and $m,n\geq 0$, we let
	$Z_m^{\pm}\subset Z$ and $Z_{m,n}^\pm\subset \A^1\times Z$ be the \emph{derived} vanishing loci of the functions
	\begin{equation*}
	Z\to\G_m\times \G_m \xrightarrow{g_m^\pm} \A^1\quad\text{and}\quad
	\A^1\times Z\to \A^1\times\G_m\times \G_m \xrightarrow{h_{m,n}^\pm} \A^1.
	\end{equation*}
	The fibers of $Z_{m,n}^\pm$ over $0$ and $1$ in $\A^1$ are $Z_m^\pm$ and $Z_n^\pm$, respectively.
	
	By \cite[Corollary 3.4]{BachmannFFlatCancellation}, if $Z$ is finite flat over $\G_m\times X$, then $Z_{m,n}^\pm$ is finite flat over $\A^1\times X$ for $m,n$ large enough. For $i\geq 0$, let 
	 \[
	 F_i^\sC(Y)\subset \Hom(\G_m,\h^\sC(\G_m\times Y))
	 \]
	 be the subpresheaf consisting of spans $\G_m\times X\leftarrow Z\to \G_m\times Y$ such that $Z_{m,n}^+$ and $Z_{m,n}^-$ are finite flat over $\A^1\times X$ for all $n,m\geq i$; this is an exhaustive filtration of $\Hom(\G_m,\h^\sC(\G_m\times X))$ in $\Pre_\Sigma(\sC)$.
	 
	 Since $Z_{m,n}^\pm$ is cut out by a single equation in $\A^1\times Z$, the closed immersion $Z_{m,n}^\pm\hookrightarrow \A^1\times Z$ is quasi-smooth with trivialized conormal sheaf. On the other hand, the projection $\A^1\times\G_m\times X\to \A^1\times X$ is smooth with trivialized cotangent sheaf. Hence, if $Z\to \G_m\times X$ is
	  framed quasi-smooth, oriented quasi-smooth, quasi-smooth, or has trivialized or invertible dualizing sheaf, then the same holds for $Z_{m,n}^\pm\to \A^1\times X$ (this is essentially the only new observation needed compared to the finite flat case).
	For $m,n\geq i$, we can therefore define
	\[
	\rho_{m}^\pm\colon F_i^\sC(Y)\to \h^\sC(Y)
	\quad\text{and}\quad
	\rho_{m,n}^\pm\colon F_i^\sC(Y)\to \h^\sC(Y)^{\A^1}
	\]
	by sending a $\sC$-correspondence $\G_m\times X\leftarrow Z\to\G_m\times Y$ to the $\sC$-correspondences
 $X\leftarrow Z_{m}^\pm\to Y$ and
	 $\A^1\times X\leftarrow Z_{m,n}^\pm\to Y$, respectively. Note that $\rho_{m,n}^\pm$ is an $\A^1$-homotopy between $\rho_m^\pm$ and $\rho_n^\pm$. The morphisms
	 \[
	 \rho_i^\pm\colon F_i^\sC(Y)\to \h^\sC(Y)
	 \]
	 and the $\A^1$-homotopies $\rho_{i,i+1}^\pm$ induce in the colimit a pair of morphisms
	 \[
	 \rho^\pm\colon \Hom(\G_m,\h^\sC(\G_m\times Y))\to \Lhtp \h^\sC(Y).
	 \]
	 We let
	 \[
	 \rho =\rho^+-\rho^-\colon \Hom(\G_m,\h^\sC(\G_m\times Y))\to \Lhtp \h^\sC(Y)^\gp.
	 \]
	 By \cite[Proposition 2.8]{BachmannFFlatCancellation}, the canonical map
	 \[
	 \Hom(\G_m,\h^\sC(\G_m\times Y))^\gp \to \Hom(\G_m,L_\mot\h^\sC(\G_m\times Y)^\gp)
	 \]
	 is a motivic equivalence, so we obtain an induced morphism
	 \[
	 \rho\colon \Hom(\G_m,L_\mot\h^\sC(\G_m\times Y)^\gp)\to L_\mot \h^\sC(Y)^\gp.
	 \]
	 
	 We now check that $\rho$ satisfies conditions (1)–(3) of \cite[Proposition 2.16]{BachmannFFlatCancellation}.
	 Conditions (1) and (2) follow from the corresponding facts about $\rho_{m,n}^\pm$, namely, the commutativity of the triangle
	\[
	\begin{tikzcd}
		\h^\sC(U)\otimes F_i(Y) \otimes \h^\sC(\A^1) \ar{dr}{\id\otimes\rho_{m,n}^\pm} \ar{d} & \\
		F_i(U\times Y)\otimes\h^\sC(\A^1) \ar{r}[swap]{\rho_{m,n}^\pm} & \h^\sC(U\times Y)
	\end{tikzcd}
	\]
	and the naturality of $\rho_{m,n}^\pm$ in $Y\in\Sm_k$.
	As in the proof of \cite[Theorem 3.5]{BachmannFFlatCancellation}, condition (3) reduces to the existence of an $\A^1$-homotopy
	\[
	\rho^+_2(\id_{\G_m}) \simeq_{\A^1} \id + \rho^-_2(\id_{\G_m})
	\]
	between endomorphisms of $\Spec k$ in $\sC$. Here, $\rho_2^+(\id_{\G_m})$ and $\rho_2^-(\id_{\G_m})$ are the framed finite syntomic $k$-schemes $\Spec k[t]/(t^2+1)$ and $\Spec k[t]/(t^2+t)$, respectively.
	Let $H\subset \A^2=\Spec k[s,t]$ be the vanishing locus of $t^2+st+1-s$. Then $H$ is framed finite syntomic over $\A^1=\Spec k[s]$ and defines an $\A^1$-homotopy from $\rho_2^+(\id_{\G_m})$ to $\id + \rho^-_2(\id_{\G_m})$, as desired.
	
	For rational contractibility, the proof of \cite[Proposition 3.8]{BachmannFFlatCancellation} applies with no significant changes. Indeed, one can replace $\h^\fflat$ by $\h^\sC$ in the statement and proof of \cite[Proposition 3.7]{BachmannFFlatCancellation}: it suffices to note that the morphism $s\colon \h^\fflat(X)\to \hat C_1\h^\fflat(X)$ constructed in \emph{loc.\ cit.}\ extends in an obvious way to a morphism $s\colon \h^\sC(X)\to \hat C_1\h^\sC(X)$.
\end{proof}

We can now prove the universality of hermitian K-theory as a generalized motivic cohomology theory with oriented finite Gorenstein transfers, in the following strong sense.

\begin{thm}\label{thm:kq-modules}
Let $k$ be a field of exponential characteristic $e \ne 2$. Then there is an equivalence of symmetric monoidal $\infty$-categories 
\[\Mod_\kq \SH(k)[\tfrac{1}{e}] \simeq  \SH^{\fgor,\o}(k)[\tfrac{1}{e}],\]
which is compatible with the forgetful functors to $\SH(k)$.
\end{thm}

\begin{proof}
	The symmetric monoidal forgetful functor $\gamma\colon \Span^\fr(\Sm_k)\to\Span^{\fgor,\o}(\Sm_k)$ gives rise to an adjunction
\[
\gamma^*: \SH(k)\simeq\SH^\fr(k) \rightleftarrows \SH^{\fgor,\o}(k) : \gamma_*
\]
where the left adjoint $\gamma^*$ is symmetric monoidal. We claim that the right adjoint $\gamma_*$ sends the unit to the $\Einfty$-algebra $\mathrm{kq}$. To prove this, since $\gamma_*$ commutes with pro-smooth base change, we can assume that $k$ is perfect. Then the cancellation and rational contractibility properties of the category of oriented Gorenstein correspondences obtained in Proposition~\ref{prop: cancellation + rat contract}, together with \cite[Corollary~2.20]{BachmannFFlatCancellation}, imply that
\[ \gamma_* (\MonUnit) \simeq \Sigma^\infty_{\T, \fr} \FGor^\o_k. \]
As we showed in Theorem~\ref{thm:KQ-kq-FGor},  $\Sigma^\infty_{\T, \fr} \FGor^\o_k \simeq \kq$.

We therefore obtain an induced adjunction
\[
\Mod_{\mathrm{kq}}(\SH(k))\rightleftarrows \SH^{\fgor,\o}(k),
\]
which we claim is an equivalence after inverting $e$. Since the right adjoint is conservative, it suffices to show that the left adjoint is fully faithful. By construction, the unit of the adjunction is an equivalence on the unit object $\mathrm{kq}$, hence on any dualizable object. But $\SH(k)[\tfrac 1e]$ is generated under colimits by dualizable objects \cite[Theorem 3.2.1]{shperf}, so the claim follows.
\end{proof}

\section{Summary of framed models for motivic spectra}
\label{sec:framed-models}

In this final section, we offer a summary of the known geometric models for the most common motivic spectra. For simplicity, we first state the results in the regular equicharacteristic case; for the state of the art see Remark~\ref{rem:general base}.

\begin{thm}\label{thm:all mot spectra}
Let $S$ be a regular scheme over a field. Then the forgetful maps of $\Einfty$-semirings 
\[\begin{tikzcd}
      & \FSyn  \ar{r} & \FFlat  \ar{r} & \Vect \ar{r} & \uZ \\
   \FSyn^\fr \ar{r} & \FSyn^\o \ar{u} \ar{r} & \FGor^\o \ar{r} \ar{u}& \Vect^\sym \ar{u}\ar{r} & \uGW\ar{u} 
\end{tikzcd}\]
induce, upon taking framed suspension spectra, canonical maps of motivic $\Einfty$-ring spectra over $S$ (assuming $2\in\sO(S)^\times$ for $\kq$):
\[\begin{tikzcd}
      & \MGL  \ar{r} & \kgl  \ar["\simeq"]{r} & \kgl \ar{r} & \hz \\
   \MonUnit \ar{r} & \MSL \ar{u} \ar{r} & \kq \ar["\simeq"]{r} \ar{u}& \kq \ar{u}\ar{r} & \hzmw \ar{u} \rlap .
\end{tikzcd}\]
\end{thm}

\begin{proof}
The functor $\Sigma^\infty_{\T,\fr}$ is symmetric monoidal, so it takes the unit $\FSyn^\fr$ in presheaves with framed transfers to the unit $\MonUnit$ in motivic spectra.
The leftmost vertical arrow is a consequence of \cite[Theorems~3.4.1 and~3.4.3]{deloop3}. The top row is contained in \cite[Corollary~5.2, Theorem~5.4]{robbery} and \cite[Theorem~21]{framed-loc}, while the bottom row follows from Proposition~\ref{prop:KQ-kq-Vect^bil}, Theorem~\ref{thm:KQ-kq-FGor}, and Theorem~\ref{thm:HtildeZ}.
\end{proof}

\begin{rem}\label{rem:general base}
	The statements about $\MonUnit$, $\MSL$, $\MGL$, and $\hz$ in Theorem~\ref{thm:all mot spectra} are in fact proved over general base schemes in the given references. 
	The statement about $\hzmw$ holds under the assumption of Theorem~\ref{thm:HtildeZ}, and the statements about $\kgl$ and $\kq$ were recently extended by Bachmann to the same generality \cite{BachmannKGL}.
\end{rem}

\begin{cor}
On the level of infinite $\P^1$-loop spaces, the diagram of motivic spectra in Theorem~\ref{thm:all mot spectra} is the motivic localization of the following diagram of forgetful maps (up to the Nisnevich-local plus construction for the left half in characteristic zero):
\[\begin{tikzcd}
   &\Z\times \Hilb_\infty^\lci(\A^\infty)  \ar{r} & \Z\times \Hilb_\infty(\A^\infty) \ar{r} & \uZ \\
   \Z\times \Hilb_\infty^\fr(\A^\infty) \ar{r} &    \Z\times \Hilb_\infty^{\lci,\o}(\A^\infty) \ar{u} \ar{r} & \Z\times \Hilb_\infty^{\Gor,\o}(\A^\infty) \ar{r} \ar{u} & \uGW\ar{u}\rlap.
\end{tikzcd}\]
\end{cor}

\begin{proof}
    This follows from Theorem~\ref{thm:all mot spectra} using \cite[Theorems 1.2 and~1.5]{deloop4}, \cite[Corollary~4.5]{robbery}, and Corollary~\ref{cor:Hilb-FGor}.
\end{proof}

Let us conclude with some comments on the ``canonical maps'' in Theorem~\ref{thm:all mot spectra}.
First, we note that the $\Einfty$-maps to $\hzmw$ and $\hz$ in the diagram are all \emph{unique}. For $\hzmw$, this is because $\hzmw=\underline\pi^\eff_0(\1)$ and the unit maps of all the spectra in the lower row induce equivalences on $\underline\pi^\eff_0$ (cf.\ \cite[Corollary~3.6.7]{MuraMSL}). Similarly, we have $\hz=s_0(\1)$ \cite[Theorem 10.5.1]{Levine:2008} and the unit maps of all spectra except $\hzmw$ induce equivalences on $s_0$ (see \cite[Remark 10.2]{SpitzweckHZ}, \cite[Example 16.35]{norms}, and \cite[Theorem 3.2]{kq-slices}). The $\Einfty$-map $\hzmw\to \hz$ is also unique since $\hz$ is in the heart of the effective homotopy $t$-structure.

The $\Einfty$-map $\MGL\to\kgl$ in Theorem~\ref{thm:all mot spectra} was shown in \cite[Proposition~6.2]{robbery} to induce the usual Thom classes in algebraic K-theory, i.e., it is an $\Einfty$ refinement of the usual orientation map obtained from the universal property of $\MGL$ as a homotopy commutative ring spectrum \cite[Theorem 2.7]{Panin:2008}. A similar argument will show that the $\Einfty$-map $\MSL\to\kq$ induces the usual Thom classes of oriented vector bundles in Grothendieck–Witt theory, once one promotes the forgetful map $\FQSm^{\o,4*}\to \Perf^\sym$ to a morphism of framed $\Einfty$-semirings (which requires some work but should present no essential difficulty). By \cite[Theorem 5.9]{PaninWalter}, it is known that there exists at least one unital morphism $\MSL\to\kq$ inducing the usual Thom classes in Grothendieck–Witt theory, but its uniqueness and multiplicativity properties are unclear.

To the best of our knowledge, the existence of $\Einfty$-maps $\MGL\to \kgl$ and $\MSL\to\kq$ was not known before the main result of \cite{deloop3}, which describes $\MGL$ and $\MSL$ as framed suspension spectra.

\let\mathbb=\mathbf
{\small
\newcommand{\etalchar}[1]{$^{#1}$}
\providecommand{\bysame}{\leavevmode\hbox to3em{\hrulefill}\thinspace}

}
\parskip 0pt

\end{document}